\documentclass[11pt,a4paper]{birkjour}

\usepackage{color}
\usepackage{geometry}
\usepackage{mathrsfs} 
\usepackage{graphicx}

\DeclareFontFamily{U}{matha}{\hyphenchar\font45}
\DeclareFontShape{U}{matha}{m}{n}{
  <-6> matha5 <6-7> matha6 <7-8> matha7
  <8-9> matha8 <9-10> matha9
  <10-12> matha10 <12-> matha12
  }{}
\DeclareSymbolFont{matha}{U}{matha}{m}{n}
\DeclareMathSymbol{\lesssim}{3}{matha}{"C0}

\newtheorem{theorem}{Theorem}[section]
\newtheorem{proposition}[theorem]{Proposition}
\newtheorem{corollary}[theorem]{Corollary}
\newtheorem{lemma}[theorem]{Lemma}
\newtheorem{remark}[theorem]{Remark}
\newtheorem{definition}[theorem]{Definition}

\newcommand{\M}{\mathcal M}

\DeclareMathOperator{\Tr}{tr}
\newcommand{\Id}{\mathrm{id}}

\newcommand{\Hs}{\mathrm H}

\newcommand{\R}{\mathbb R}
\newcommand{\bfX}{\mathbf X}

\newcommand{\eps}{\varepsilon}

\newcommand{\Nabla}{\overline\nabla}

\begin{document}

\title[Kinetic Brownian motion in configuration space]{Kinetic Brownian motion on the diffeomorphism group of a closed Riemannian manifold}

\author{\textsc{J. Angst}}\address{Univ Rennes, CNRS, IRMAR - UMR 6625, F-35000 Rennes, France}\email{jurgen.angst@univ-rennes1.fr}
\author{\textsc{ I. Bailleul}}\address{Univ Rennes, CNRS, IRMAR - UMR 6625, F-35000 Rennes, France}\email{ismael.bailleul@univ-rennes1.fr}
\author{\textsc{P. Perruchaud}}\address{Univ Rennes, CNRS, IRMAR - UMR 6625, F-35000 Rennes, France}\email{pierre.perruchaud@univ-rennes1.fr}

\thanks{The authors thank A. Kulik for helful conversations on ergodic properties of Markov processes. I.Bailleul thanks the U.B.O. for their hospitality, part of this work was written there.}

\begin{abstract}
We define kinetic Brownian motion on the diffeomorphism group of a closed Riemannian manifold, and prove that it provides an interpolation between the hydrodynamic flow of a fluid and a Brownian-like flow.
\end{abstract}

\maketitle

\setcounter{tocdepth}{2}
{\small 
\tableofcontents
}

\section{Introduction}
\label{sec.introduction}

Kinetic Brownian motion is a purely geometric random perturbation of geodesic motion. In its simplest form, in $\mathbb{R}^d$, the sample paths of kinetic Brownian motion are $C^1$ random paths run at unit speed, with velocity a Brownian motion on the unit sphere, run at speed $\sigma^2$, for a speed parameter $\sigma>0$. More formally, it is a hypoelliptic diffusion with state space $\mathbb{R}^d\times\mathbb{S}^{d-1}$, solution to the stochastic differential equation
\begin{equation*}\begin{split}
dx^\sigma_t &= v^\sigma_t\,dt,   \\
dv^\sigma_t &= \sigma\,P_{v^\sigma_t}({\circ dW_t}),
\end{split}\end{equation*}
for $P_a : \mathbb{R}^d\rightarrow \langle a\rangle^\perp$, the orthogonal projection on the orthogonal of $\langle a\rangle$, for $a\neq 0$ in $\mathbb{R}^d$, and $W$ a standard $\mathbb{R}^d$-valued Brownian motion. If $\sigma=0$, we have a straight line motion with constant velocity. For a fixed $0<\sigma<+\infty$, we have a $C^1$ random path, whose typical behavior is illustrated in Figure \ref{fig.1} below.\par
\begin{figure}[ht]\label{fig.1}
\centering
\includegraphics[scale=.25]{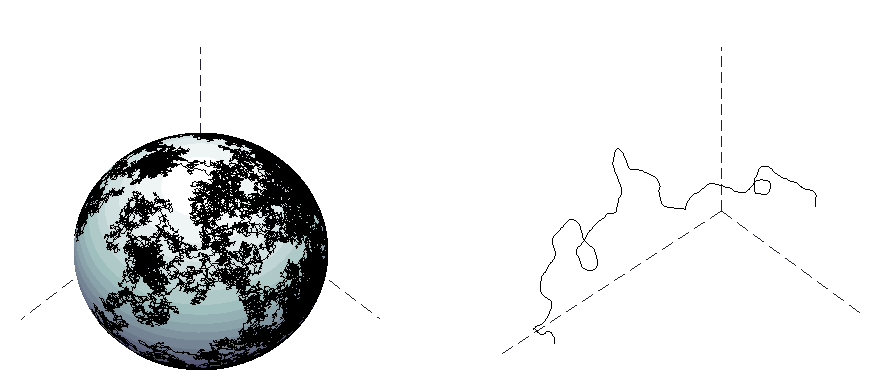}
\caption{Brownian motion on the sphere and its integral path in $\mathbb{R}^d$.}
\end{figure} 
For $\sigma$ increasing to $\infty$, the exponentially fast decorrelation of the velocity process $v^\sigma$ on the sphere implies that the process $x^\sigma$ converges to the constant path $x_0$, if the latter is fixed independently of $\sigma$. One has to rescale time and look at the evolution at the time scale $\sigma^2$ to see a non-trivial limit. It is indeed elementary to prove that the time rescaled position process $(x^\sigma_{\sigma^2t})_{0\leq t\leq 1}$ of kinetic Brownian motion converges weakly in $C\big([0,1],\mathbb{R}^d\big)$ to a Brownian motion with generator $\frac{4}{d(d-1)}\,\Delta_{\mathbb R^d}$. See Figure 2 below for an illustration in the setting of the flat $2$-dimensional torus. This homogenization result is in fact valid on a general finite dimensional Riemannian manifold $M$, under very mild geometric assumptions.
\begin{figure}[ht]
\label{Picture2}
\centering
\includegraphics[scale=0.25]{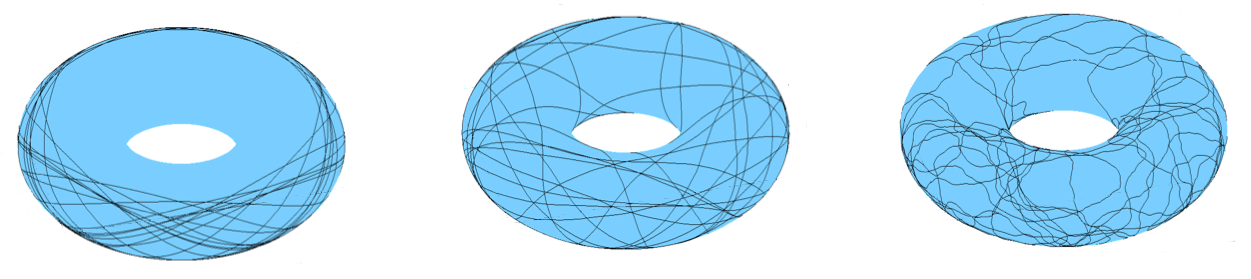}
\includegraphics[scale=0.25]{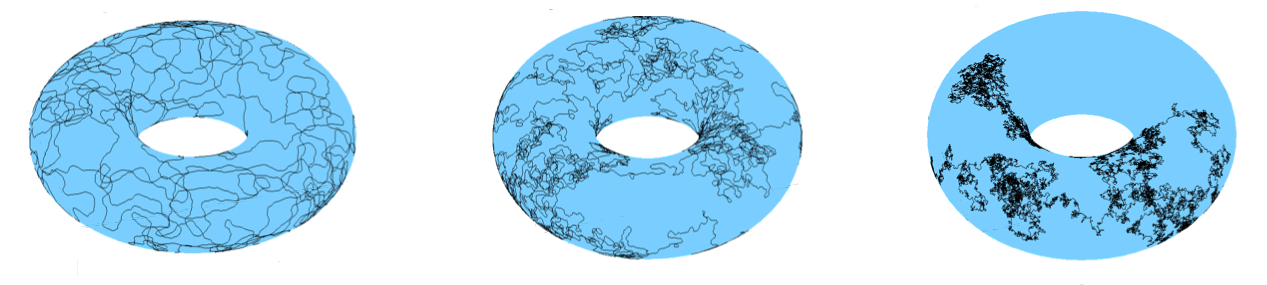}
\caption{Sample paths of kinetic Brownian motion $\big(x^\sigma_{\sigma^2t}\big)_{0\leq t\leq 1}$ as $\sigma$ increases.}
\end{figure} 
Kinetic Brownian motion on a $d$-dimensional Riemannian manifold $M$ is defined as Cartan development $\big(m^\sigma_t,\dot m^\sigma_t\big)$ in the unit tangent bundle $T^1M$ of $M$ of kinetic Brownian motion in $\mathbb{R}^d$. It is a geodesic for $\sigma=0$, and a $C^1$ random path for a finite positive value of $\sigma$. It was first proved by X.-M. Li in \cite{XueMei0} that the time-rescaled position process $(m^\sigma_{\sigma^2t})_{0\leq t\leq 1}$ converges weakly to Brownian motion with generator $\frac{4}{d(d-1)}\,\Delta_M$. The manifold $M$ was assumed to be compact and martingale methods were used to prove that homogenization result. X.-M. Li extended this result in \cite{XueMei1} to non-compact manifolds subject to a growth condition on their curvature tensor. In \cite{ABT}, Angst, Bailleul and Tardif gave the most general result, assuming only geodesic and stochastic completeness, using rough paths theory as a working horse to transport a rough path convergence result about kinetic Brownian motion in $\mathbb{R}^d$ to the manifold setting. See also \cite{XueMei2} for further results in homogeneous spaces, and \cite{PierreFiniteDim} for a generalization of the homogenization result of \cite{ABT} to anisotropic kinetic Brownian motion, or more general Markov processes on $T^1M$. Note that the dynamically obvious convergence of the unrescaled kinetic Brownian motion to the geodesic motion has been studied from the spectral point of view in \cite{Drouot}, for compact manifolds with negative curvature, showing that the $L^2$ spectrum of the generator of the unrescaled kinetic Brownian motion converges to the Pollicott-Ruelle resonances of $M$. Other examples of homogenization results for Langevin-type processes include works by Hottovy and co-authors, amongst others; see e.g. \cite{BirrellHottovyVolpeWehr, HerzogHottovyVolpe, BirrellWehr, LimWehrLewenstein} for quantitative convergence results. See also \cite{Soloveitchik, Kolokoltsov, VassiliHilbert, Gliklikh} for other works on Langevin dynamics in a Riemannian manifold.

\smallskip

This kind of homogenization result certainly echoes Bismut's program about his hypoelliptic Laplacian \cite{Bismut05, Bismut}, whose probabilistic starting point is a similar interpolation result for Langevin process in $\mathbb{R}^d$ and its Cartan development on a Riemannian manifold. The dynamics is lifted to a dynamics on the space of differential forms to take advantage of the correspondence between the cohomology of differential forms and homology of $M$, via index-type theorems. See \cite{BismutOrbital, Bismut, BismutEta, ShuShen} for a sample of the deep results obtained by Bismut and co-authors on the hypoelliptic Laplacian. 

Note also that kinetic Brownian motion is the Riemannian analogue of its Lorentzian counterpart, introduced first by Dudley in \cite{Dudley} in Minkowski spacetime in the 60's. See the far reaching related works \cite{FranchiLeJanRelDiff1, BailleulRelDiff, FranchiLeJanRelDiff2, BailleulFranchi}, on relativistic diffusions in a general Lorentzian setting. No homogenization result is expected for these purely geometric diffusion processes, unless one has an additional non-geometric ingredient, e.g. in the form of a relativistic fluid flow, like in \cite{AngstFranchi}.

\smallskip

The object of the present work is to define and study kinetic Brownian motion in the diffeomorphism group $\mathscr{M}$, or volume preserving diffeomorphism group $\mathscr{M}_0$, of a closed Riemannian manifold $M$. As in the finite dimensional setting, we prove that it provides an interpolation between the geodesic flow and a Brownian flow, as the noise intensity parameter $\sigma$ ranges from $0$ to $\infty$. For $\sigma=0$, the motion in each diffeomorphism group is geodesic, and it corresponds to the flow of the solutions of Euler's equation in the case of $\mathscr{M}_0$, after the seminal works of Arnold \cite{Arnold} and Ebin \& Marsden \cite{EbinMarsden}. When considered in the setting of volume preserving diffeomorphisms, the Eulerian picture of kinetic Brownian motion provides a family of random perturbations of Euler's equations for the hydrodynamics of an incompressible fluid. There has been much work recently on random perturbations of Euler's equations, following Holm's seminal article \cite{Holm15}. See \cite{GayBalmazHolm, CruzeiroHolmRatiu, CrisanFlandoliHolm, DrivasHolmCirculationCharact, KIWFormulaHolm} for a sample. The structure of the noise in these works is intrinsically linked to the group structure of the diffeomorphism group, and it amounts to perturbe Euler's equation for the velocity field by an additive Brownian term, with values in a space of vector fields on the fluid domain $M$. Our point of view is purely Riemannian, and does not appeal to the group structure of the diffeomorphism group of the fluid domain $M$. As in the above finite dimensional setting, we define kinetic Brownian on the diffeomorphism group as the Cartan development of its `flat' counterpart. Unlike the group-oriented point of view, where the running time diffeomorphism is sufficient to describe its infinitesimal increment from the noise, we need here a notion of frame of the tangent space of the running diffeomorphism to build its increment from the noise. We prove that each component of the energy spectrum of the Eulerian velocity field is ergodic, and give an explicit description of its invariant measure. We also have an explicit description of the invariant measure of the energy of the Eulerian velocity field.   

\smallskip

On the technical side, we use rough paths theory to transport a weak convergence result for the flat kinetic Brownian motion taking values in the tangent space to the configuration space $\mathscr{M}$, or $\mathscr{M}_0$, to a weak convergence result for the solution of a differential equation controlled by that flat kinetic Brownian motion. We use for that purpose the continuity of the It\^o-Lyons solution map to a controlled ordinary differential equation, in the present infinite dimensional setting. This allows to bypass a number of difficulties that would appear otherwise if using the classical martingale problem approach, as in \cite{XueMei0, XueMei1}. All we need about rough paths theory is recalled in Section \ref{SubsectionRP}.

\smallskip

From a geometric point of view, the tangent space to the configuration space can naturally be seen as an infinite dimensional Hilbert space.  For this reason, we define and study in Section \ref{SectionHilbertKBM} kinetic Brownian motion on a generic infinite dimensional Hilbert space $H$. We provide an explicit description of the invariant measure of the velocity process in Section \ref{SubsectionBMSphere}, and we establish exponential decorrelation identities for the latter in Section \ref{SubsectionMixing}. The invariance principle for the position process associated to the time-rescaled $H$-valued kinetic Brownian motion is then established in Section \ref{SubsectionCvgceX}. With the rough paths tools introduced in Section \ref{SubsectionRP}, Section  \ref{SubsectionCvgceRP} is devoted to the proof of the fact that the canonical rough path above the time-rescaled position process converges weakly as a rough path to the Stratonovich Brownian rough path of a Brownian motion with an explicit covariance. Elements of the geometry of the configuration spaces $\mathscr{M}$ and $\mathscr{M}_0$ are recalled in Section \ref{SectionGeometry}. We develop in particular in Section \ref{SubsectionParallelTransport} and Section \ref{SubsectionCartan} the material needed to talk about Cartan development operation as solving an ordinary differential equation driven by smooth vector fields. The final homogenisation result, proving the interpolation between geodesic and Brownian flows on the configuration spaces, is proved in Section \ref{SectionKBM} using the robust tools of rough paths theory. Appendix \ref{AppendixCartanM0} contains the proof of a technical result about Cartan development in $\mathscr{M}_0$.

\bigskip

\noindent \textbf{Notations.} We gather here a number of notations that are used throughout the article.

\smallskip

\begin{itemize}
   \item The letter $\gamma$ stands for a Gaussian measure on a Hilbert space $H$, with covariance $C_\gamma : H^*\times H^* \rightarrow \mathbb{R}$, and associated operator $\overline{C}_\gamma : H\rightarrow H$. The scalar product and norm on $H$ are denoted by $(\cdot, \cdot)$ and $\|\cdot\|$, respectively. \vspace{0.1cm}

   \item We denote by $\mathcal{H}$ the Cameron-Martin space of the measure $\gamma$.   \vspace{0.1cm}

   \item We endow the algebraic tensor space $H\otimes_a H$ with its natural Hilbert norm. This amounts to identify $H\otimes H$ with the space of Hilbert-Schmidt operators on $H$.   \vspace{0.1cm}
   
   \item We use the notation $A\lesssim_p B$ for an inequality of the form $A\leq cB$, with a constant $c$ depending only on $p$.
\end{itemize}

\section{Kinetic Brownian motion in a Hilbert space}
\label{SectionHilbertKBM}

\subsection{Brownian motion on a Hilbert sphere}
\label{SubsectionBMSphere}

We first recall basic results on Brownian motion in $H$, and refer the reader to the nice lecture notes \cite{LNHairerSPDE, Stroock} for short and detailed accounts. 

\smallskip

Recall that a \textbf{Gaussian probability measure on $H$} is a Borel measure $\gamma$ such that $\ell^*\gamma$ is a real Gaussian probability on $\mathbb{R}$, for every continuous linear functional $\ell : H \rightarrow\mathbb{R}$. Fernique's theorem \cite{Fernique} ensures that 
$$
\int_H\exp\big(a\|x\|^2\big)\,\gamma(dx) < \infty,
$$
for a small enough positive constant $a$. It follows that the covariance 
$$
C_\gamma(\ell,\ell') := \int \ell(x)\ell'(x)\,\gamma(dx), \quad \ell,\ell'\in H^*
$$
is a well-defined continuous bilinear operator on $H^*\times H^*$. One can then define a continuous symmetric operator $\overline{C}_\gamma : H\rightarrow H$, by the identity 
$$
\big(\overline{C}_\gamma(h),k\big) = C(h,k),
$$ 
for all $h,k\in H$. It has finite trace equal to
$$
\textrm{tr}(\overline{C}_\gamma) = \int \|x\|^2\,\gamma(dx).
$$
Conversely, one can associate to any trace-class symmetric operator $\overline{C} : H\rightarrow H$, a Gaussian measure $\gamma$ on $H$ whose covariance $C_\gamma(\ell,\ell) = \overline{C}(\ell,\ell)$, for all $\ell\in H$. Since $\overline{C}_\gamma$ is compact, there exists an orthonormal basis $(e_n)$ of $H$, such that 
$$
\overline{C}_\gamma(e_n) = \alpha_n^2 e_n, 
$$
for non-negative and non-increasing eigenvalues $\alpha_n$ with $\sum \alpha_n^2 < \infty$. We define a Hilbert space $\mathcal{H}$ by choosing $\big(\alpha_ne_n\big)$ as an orthonormal basis for it. The space $\mathcal{H}$ is continuously embeded inside $H$. Let $(X^n)$ stand for a sequence of independent, identically distributed, real-valued Gaussian random variables with zero mean and unit variance, defined on some probability space $(\Omega,\mathcal{F},\mathbb{P})$. Then the series 
$$
\sum_n X^n \alpha_n e_n
$$
converges in $L^2(\Omega, H)$, and has distribution $\gamma$.

\smallskip

Fix a positive time horizon $T\in (0,\infty]$. An \textbf{$\mathcal{H}$-Brownian motion in $H$}, on the time interval $[0,T)$ is a random $H$-valued continuous path $W$ on $[0,T)$, with stationary, independent increments such that the distribution of $W_1$ is a Gaussian probability measure $\gamma$ on $H$. A simple construction is provided by taking a sequence $(W^n_t)$ of independent, identically distributed, real-valued Brownian motions, and setting
$$
W_t := \sum_n W^n_t \alpha_n e_n.
$$
Denote by $S$ the unit sphere of $H$, and let 
$$
P_a : H\rightarrow H
$$ 
stand for the orthogonal projection on $\langle a\rangle^\perp$, for $a\neq 0$. The $\mathcal{H}$-\textbf{spherical Brownian motion} $v^\sigma_t$ on $S$ is defined as the solution to the Stratonovich stochastic differential equation
\begin{equation}\label{def.vS}
dv^\sigma_t = \sigma\,P_{v^\sigma_t}({\circ dW_t})
\end{equation}
associated to a given initial condition $v_0^\sigma\in S$; it is defined for all times. The speed parameter $\sigma$ is a non-negative real number. Write $Z$ for $\int_H \frac1{\|u\|}\gamma(du)$.

\begin{theorem}   \label{PropInvariantMeasure}
The image under the projection $u\mapsto u/\|u\|$ of the measure $\frac{1}{Z}\frac1{\|u\|}\gamma(du)$ in the ambiant space $H$ is a probability measure $\mu$ on $S$ that is invariant for the dynamics of $v^\sigma_t$, for any positive speed parameter $\sigma$.
\end{theorem}

This statement generalizes Proposition 1.1 of \cite{PierreFiniteDim} to the present infinite dimensional setting. The above description of the invariant measure $\mu$ as an image measure under the projection map actually coincides with the finite dimensional description given in the latter reference.

\begin{proof}
When written in It\^o form, the stochastic differential equation \eqref{def.vS} defining the process $(v^\sigma_t)_{t \geq 0}$ reads 
\begin{equation} \label{eq.ito}
d v_t^{\sigma} = -\frac{\sigma^2}{2}  \Big( \Tr(\overline{C}_\gamma) v_t^{\sigma} + \overline{C}_\gamma(v_t^{\sigma}) - 2 C_{\gamma}(v_t,v_t) v_t^{\sigma}\Big)dt + \sigma \,P_{v^\sigma_t}({dW_t}),
\end{equation}
and setting $v^{\sigma,i}_t:=\langle v^{\sigma}_t, e_i \rangle$, for any integer $i$, we have
\begin{equation*}\begin{split}
d v^{\sigma,i}_t &= -\frac{\sigma^2}{2}  \left[ \sum_{n} \alpha_n^2+\alpha_i^2  - 2\sum_{n} \alpha_n^2 |v^{\sigma,n}_t|^2  \right] v^{\sigma,i}_t \, dt   \\
&\quad+\sigma \left[  \alpha_i d \, W^i_t- v^{\sigma,i}_t \sum_n \alpha_n v^{\sigma,n}_t d W^n_t\right].
\end{split}\end{equation*}
As in the finite dimensional anisotropic case treated in \cite{PierreFiniteDim}, it is actually easier to work with an $H$-valued lift of this $S$-valued process. We introduce for that purpose the process $(u_t^{\sigma})_{t \geq 0}$ solution of the Stratonovich stochastic differential equation
$$
d u_t^{\sigma} = - \frac{\sigma^2}{2} \| u_t^{\sigma} \|^2 u_t^{\sigma} dt + \sigma \| u_t^{\sigma}\| {\circ d}W_t;
$$
equivalently, in It\^o form and coordinate-wise, setting  $u^{\sigma,i}_t:=\langle u^{\sigma}_t, e_i \rangle$  as above, we have
$$
d u_t^{\sigma,i} = \frac{\sigma^2}{2}  \left( -  \| u_t^{\sigma}\|^2 + \alpha_i^2 \right) u_t^{\sigma,i} dt + \sigma \| u_t^{\sigma}\| \alpha_i d W_t^i.
$$
A direct application of It\^o's formula then shows that $u^{\sigma,i}_t/\| u_t^{\sigma} \|$ satisfies the same stochastic differential equation as $v^{\sigma,i}_t$, for all $i$, so the two $S$-valued processes $(v^\sigma_t)_{t \geq 0}$ and $(u^\sigma_t/\|u^\sigma_t\|)_{t \geq 0}$ have the same distributions. As in the finite dimensional case, one can then check by a direct computation that the measure $\|u\|^{-1}\gamma(du)$ on $H$ is invariant for the processes $(u^\sigma_t)$; this implies the statement of Theorem \ref{PropInvariantMeasure}.

\smallskip

Alternatively, one can bypass computations and argue using Malliavin calculus as follows. Denote by $L$ the infinitesimal generator of the process $(u^\sigma_t)$.  Set $V(u):=u/\|u\|^2$ for $u\neq 0$, and let $\Delta_{\gamma}$ denote the Laplace operator associated with the covariance $C_{\gamma}$ with weights $(\alpha_n^2)$,. We then have for any test function $f$ and any $u \in H$ 
$$
Lf(u) = \frac{\sigma^2}{2}  \| u \|^2 (L_0 f)(u),
$$
with
$$
(L_0 f)(u) := \Delta_{\gamma} f(u) - u \nabla f(u) +  C_{\gamma}\big( V(u), \nabla f(u)\big).
$$
One then has for any test function $f$, with usual notations $D$ for the gradient and $\delta$ for the divergence,
\begin{equation*}\begin{split}
\int_H Lf(u) \|u\|^{-1}\gamma(du) &= \frac{\sigma^2}{2} \int_H L_0 f(u) \|u\| \gamma(du)   \\
&= \sigma^2\,\mathbb E\Big[\big( -\delta D F + \langle V, DF \rangle_{C_\gamma}\big) \, \|W\| \Big]   \\
&= \mathbb E\Big[ \big( -\delta \underbrace{D \|W\|}_{= \frac{W}{\|W\|}}  + \delta  \underbrace{V \|W\|}_{= \frac{W}{\|W\|}}  \big)  F \Big]=0.
\end{split}\end{equation*}
\end{proof}

We prove in Section \ref{SubsectionMixing} that the velocity process $(v^\sigma_t)$ converges exponentially fast in Wasserstein distance to the invariant probability measure $\mu$ of Theorem \ref{PropInvariantMeasure}, for any initial velocity $v_0$, despite the possible lack of strong Feller property of the associated semigroup. An invariance principle for the time-rescaled position process $(x^\sigma_{\sigma^2t})$ is obtained as a consequence in Section \ref{SubsectionCvgceX}. We recall in Section \ref{SubsectionRP} what we need from rough paths theory in this work, and prove in Section \ref{SubsectionCvgceRP} that the canonical rough path associated to the time-rescaled process $(x^\sigma_{\sigma^2t})$ converges weakly as a rough path to an explicit Stratonovich Brownian rough path.

\subsection{Exponential mixing of the velocity process}
\label{SubsectionMixing}

We consider in this section the mixing properties of the spherical process $(v^\sigma_t)_{t \geq 0}$ with unit speed parameter $\sigma=1$. To simplify the expressions, we drop momentarily the exponents $\sigma$ from all our notations. Our objective is to show that the spherical process 
$$
(v_t)_{t \geq 0}= (v^1_t)_{t \geq 0}
$$ 
is exponentially mixing. Recall that the $1$ and $2$-Wasserstein distances are defined for any probability measures $\mu,\nu$ on $S$ by the identities
\begin{equation*}\begin{split}
\mathcal{W}_2(\lambda,\nu) &= \inf\Big\{\mathbb E\big[\|X-Y\|^2\big] ; X\sim\lambda, Y\sim\nu\Big\},   \\
\mathcal{W}_1(\lambda,\nu) &= \inf\Big\{\mathbb E\big[\|X-Y\|\big] ; X\sim\lambda, Y\sim\nu\Big\}   \\
											   &= \sup\left\{\int f\,d(\lambda - \nu) ; |f|_\mathrm{Lip}\leq1\right\}, 
\end{split}\end{equation*}
where the infimum is taken over all couplings $\mathbb P$ of $X\sim\lambda$ and $Y\sim\nu$, and the supremum over all $1$-Lipscthiz functions $f: S\to\R$. The first two equalities are definitions, the last one is the Kantorovich-Rubinstein duality principle. Note that $\mathcal{W}_1 \leq \mathcal{W}_2$.

\begin{proposition}   \label{prop.exponentialmixing}
Assume that 
\begin{equation}   \label{EqTraceCondition}
3\alpha_0^2<\Tr(\overline{C}_{\gamma}).
\end{equation}
There exists a positive time $\tau$ such that for any probability measures $\lambda$ and $\nu$ on the unit sphere $S$ of $H$, we have 
$$
\mathcal W_2(P_t^*\lambda,P_t^*\nu) \leq e^{-t/\tau}\mathcal W_2(\lambda,\nu),
$$
for all $t \geq 0$. In particular, the invariant measure $\mu$ is unique, and for any probability measure $\lambda$ on the sphere $S$, and $t \geq 0$, we have
\begin{equation}   \label{EqExponentialConvergence}
\mathcal W_2(P_t^*\lambda,\mu) \leq 2 e^{-t/\tau} .
\end{equation}
\end{proposition}

The role of the trace condition \eqref{EqTraceCondition} will be clear from the proof. If we have the freedom to choose the covariance $C_\gamma$ of the Brownian noise, this is not a constraint. Note that the rougher the noise, that is the more slowly the sequence of the eigenvalues $\alpha_n$ converges to $0$, the easier it is to satisfy condition \eqref{EqTraceCondition}. We shall see in Section \ref{SectionKBM} that it holds automatically in a number of relevant examples of random dynamics in the configuration space of a fluid flow.

\begin{proof}
Denote by $\mathbb P$ the law of the Brownian motion $(B_t)$ with covariance $C_{\gamma}$, and by $\mathbb P_v$ the law of the solution of Equation \eqref{def.vS}  with $\sigma=1$, starting from $v\in S$. Denote by $\mathbb E$ and $\mathbb E_v$ the associated expectations operators. Recall that the notation $(a,b)$ stands for the scalar product of $a$ and $b$ in $H$. Fix $v_0, w_0\in  S$, and consider the two diffusion processes $(v_t)$ and $(w_t)$, started from $v_0$ and $w_0$, respectively, and solutions of the It\^o stochastic differential equations
\begin{equation*}\begin{split}
d v_t &= \displaystyle{-\frac{1}{2}  \Big(\Tr(\overline{C}_\gamma) v_t + \overline{C}_\gamma(v_t) - 2 C_{\gamma}(v_t,v_t) v_t\Big)dt +  \,P_{v_t}({dW_t})},   \\
d w_t &= \displaystyle{-\frac{1}{2}  \Big(\Tr(\overline{C}_\gamma) w_t + \overline{C}_\gamma(w_t) - 2 C_{\gamma}(w_t,w_t) w_t\Big)dt +  \,P_{w_t}({dW_t})}.
\end{split}\end{equation*}
Comparing with Equation \eqref{eq.ito}, it is clear that $(v_t)$ has law $\mathbb P_{v_0}$ and $(w_t)$ has law $\mathbb P_{w_0}$. Moreover, It\^o's formula yields 
\begin{equation*}\begin{split}
d (v_t, w_t) &= \Big( \Tr(\overline{C}_\gamma) - C_{\gamma}(v_t,v_t) - C_{\gamma}(w_t,w_t) - C_{\gamma}(v_t,w_t) \Big)\big( 1- (v_t, w_t) \big)dt   \\
&\quad+ \big( 1- (v_t, w_t) \big) \Big( (v_t, dW_t) +  (w_t, dW_t)\Big),
\end{split}\end{equation*}
or equivalently, setting 
$$
N_t := \frac12\|w_t-v_t\|^2=1-( v_t,w_t),
$$ 
we get 
\begin{equation}\begin{split}   \label{EqComputation}
d N_t & = - \Big(  \Tr(\overline{C}_\gamma) - C_{\gamma}(v_t,v_t) - C_{\gamma}(w_t,w_t) - C_{\gamma}(v_t,w_t) \Big) N_t dt   \\
&\quad- N_t \big( (v_t, dW_t) +  (w_t, dW_t)\big).
\end{split}\end{equation}
Now remark that since the sequence $(\alpha_n)$ is non-increasing, we have 
$$
C_{\gamma}(v,v) = \sum_{n\geq 0} \alpha_n^2 |v_n|^2 \leq \alpha_0^2,
$$ 
for any $v \in S$. Taking the expectation under $\mathbb P$ in equation \eqref{EqComputation}, we have from Gr\"onwall inequality
$$
\mathbb E[N_t] \leq e^{-t(\Tr(C_\gamma)-3\alpha_0^2)} \mathbb E[N_0], 
$$
that is 
$$
\mathbb E[\|v_t -w_t\|^2]  \leq e^{-t(\Tr(C_\gamma)-3\alpha_0^2)}\, \|x-y\|^2.
$$
The conclusion of the statement follows.
\end{proof}

Remark that $\mathbb E_\mu[v_t] = 0$, as a consequence of the symmetry properties of the invariant measure $\mu$.

\begin{corollary}   \label{CorExpConvergence}
For any $v_0\in S$, we have 
$$
\big\|\mathbb E_{v_0}[v_t]\big\| \leq  2 e^{-t/\tau}.
$$
\end{corollary}

The process $(v_t)$ is stationary if $v_0$ has distribution $\mu$; it can then be extended into a two sided process defined for all real times. Denote by $(\mathcal{F}_t)_{t\in\mathbb{R}}$ the complete filtration generated by $(v_t)$ on the probability space where it is defined. Set $\mathcal{F}_{\leq 0} := \sigma\big(\mathcal{F}_t\,;\,t\leq 0\big)$ and $\mathcal{F}_{\geq s} := \sigma\big(\mathcal{F}_t\,;\,t\geq s\big)$, for any real time $s$. Recall that the mixing coefficient $\alpha(s)$ of the velocity process $v$ is defined, for $s>0$, by the formula
$$
\alpha(s) := \sup_{A\in\mathcal{F}_{\leq 0}, B\in\mathcal{F}_{\geq s}} \big|\mathbb{P}(A\cap B)-\mathbb{P}(A)\mathbb{P}(B)\big|.
$$
The following fact will be useful to get for free the independence of the increments of the limit processes obtained after proper rescalings of functionals of $(v_t)$.

\begin{corollary}   \label{CorMixingCoefficient}
The mixing coefficient $\alpha(s)$ tends to $0$ as $s$ increases to $\infty$.
\end{corollary}

\begin{proof}
As a preliminary remark, recall the definition of the lift $(u^\sigma_t)$ to $H$ of $(v^\sigma_t)$, introduced in the proof of Theorem \ref{PropInvariantMeasure}. This process is strong Feller, as it can be seen to satisfy a Bismut-Li integration by parts formula. See e.g. Peszat and Zabczyk' seminal paper \cite{PeszatZabczyk}, and Wang and Zhang's extension \cite{WangZhang} to unbounded drift and diffusivity. The velocity process $(v^\sigma_t)$ is thus itself a strong Feller diffusion, and if one denotes by $(P_t)$ its transition semigroup, the functions $P_1g$, for $g$ measurable, bounded by $1$, are all Lipschitz continuous, with a finite common upper bound $L$ for their Lipschitz constants.

Now, it follows from the Markovian character of the dynamics of $(v_t)$, and the Feller property of its semigroup, that it suffices to see that 
\begin{equation}   \label{EqMixingCoefficient}
\mathbb{E}\big[f(v_0)g(v_s)\big]
\end{equation}
tends to $0$ as $s$ goes to $\infty$, for any real-valued continuous functions $f,g$ on the unit sphere $S$, with null mean with respect to the invariant measure $\mu$. Writing further 
$$
\mathbb{E}\Big[f(v_0)\mathbb{E}\big[g(v_s)|v_{s_1}\big]\Big] = \mathbb{E}\Big[f(v_0)\,(P_1g)(v_{s-1})\Big],
$$
for $s>1$, and using the strong Feller property of the semigroup of the diffusion process $(v_t)$, we can further assume that the function $g$ in \eqref{EqMixingCoefficient} is $L\|g\|_\infty$-Lipschitz continuous. Let $w_g$ stand for its uniform modulus of continuity. For each $s$, denote by $(v_s,\overline{v}_s)$ a $\mathcal{W}_1$-optimal coupling of the measures $P_s^*\delta_{v_0}$ and $\mu$, for a deterministic $v_0$, so we have
$$
\mathbb{E}\big[|v_s-\overline{v}_s|\big] = \mathcal{W}_1\big(P_s^*\delta_{v_0},\mu\big).
$$
Using the fact that $\int gd\mu=0$, one then has
\begin{equation*}\begin{split}
\big|\mathbb{E}\big[f(v_0)g(v_s)\big]\big| &= \Big|\mathbb{E}\Big[f(v_0)\,\mathbb{E}\big[g(v_s)|v_0\big]\Big]\Big|   \\
&\leq \|f\|_\infty\, \mathbb{E}\Big[w_g\big(|v_s-\overline{v}_s|\big)\Big]   \\
&\leq L\|f\|_\infty\|g\|_\infty\, \mathbb{E}\big[|v_s-\overline{v}_s|\big],
\end{split}\end{equation*}
so the statement follows from Proposition \ref{prop.exponentialmixing}.
\end{proof}

\subsection{Invariance principle for the position process}
\label{SubsectionCvgceX}

\textit{We assume in all of this section that the initial condition $v_0$ of the velocity process of kinetic Brownian motion is distribued according to its invariant probability measure $\mu$, from Theorem \ref{PropInvariantMeasure}.}

\medskip

Pick $1/3<\alpha\leq 1/2$. We prove in this section that the distribution in $C^\alpha([0,1],H)$ of the time-rescaled position process $(x_{\sigma^2t}^{\sigma})$ converges to the distribution of a Brownian motion in $H$ with an explicit covariance, given in identity \eqref{EqCovarianceLimitBM} of Proposition \ref{prop.convergenceHolder} below. The usual invariance principles in Hilbert spaces consider weak convergence in $C^0([0,1],H)$, so we need an extra tightness estimate provided in Section \ref{sssec.tighnessHolder} to complete the program. To make the most out of the convergence results from Section \ref{SubsectionMixing}, set 
$$
X^\sigma_t := x^\sigma_{\sigma^2t};
$$ 
we have
$$
X^\sigma_t - X^\sigma_s = \int_{\sigma^2s}^{\sigma^2t}v_{\sigma^2u} d u= \frac1{\sigma^2}\int_{\sigma^4s}^{\sigma^4t}v_u d u, 
$$
with $(v_t) = (v^1_t)$, the spherical Brownian motion run at speed $\sigma^2=1$.

\begin{proposition}
\label{prop.convergenceHolder}
For every $0<\alpha<1/2$, the distribution in $\mathcal C^\alpha([0,1], H)$ of the process $(X^\sigma_t)$ converges as $\sigma$ goes to $\infty$ to the Brownian motion on $H$ with covariance operator
\begin{equation}   \label{EqCovarianceLimitBM}
C(\ell,\ell') := \int_0^\infty\mathbb{E}\Big[\ell(v_0)\ell'(v_t)+\ell'(v_0)\ell(v_t)\Big]\,dt,
\end{equation}
for $\ell,\ell'\in H^*$.
\end{proposition}

\subsubsection{Tightness in H\"older spaces}
\label{sssec.tighnessHolder}

We dedicate this section to proving the following uniform estimate.

\begin{proposition}
\label{lem.kolmogorovwk}
For any $p\geq2$, we have
$$
\sup_{\sigma>0}\,\mathbb E\big[\|X^\sigma_t-X^\sigma_s\|^p\big] \lesssim_p |t-s|^{p/2}. 
$$
\end{proposition}

It follows from Kolmogorov-Lamperti tightness criterion that the laws of $X^\sigma$ form a tight family in $\mathcal C^\alpha([0,1],H)$, for any $0<\alpha<1/2$. Note that for $T=\sigma^4(t-s)>0$, we have
$$
\big\|X^\sigma_t-X^\sigma_s\big\| \overset{\mathcal{L}}{=} \frac1{\sigma^2}\left\|\int_0^{\sigma^4(t-s)}v_u\,du\right\| = |t-s|\cdot\frac1{\sqrt T}\left\|\int_0^Tv_u\,du\right\|, 
$$
so Proposition \ref{lem.kolmogorovwk} is a consequence of the estimate
$$
\mathbb E\left[\left|\int_0^Tv_t\,dt\right|^p\right]\lesssim_p T^{p/2}. 
$$
We translate our problem in discrete time, writing
$$
\int_0^T = \sum_{k<T}\int_k^{k+1}
$$
to work with the correlations between different integral slices, and compare this sequence to martingale differences. There is an abundant literature on the subject; we follow here the approach of C. Cuny \cite{Cuny}.

\medskip

Let $(\Omega,\mathcal F,\mathbb P)$ be a probability space with a filtration $(\mathcal F_n)_{n\geq n_0}$, where $-\infty\leq n_0\leq0$, and let $(X_n)_{n\geq n_0}$ be $H$-valued random variables such that each $X_n$ is measurable with respect to $\mathcal F_n$. Recall that $(X_n)_{n\geq0}$ is said to be a \textbf{martingale difference} with respect to $(\mathcal F_n)$ if each $X_n$ is integrable and $\mathbb E\big[X_{n+1}|\mathcal F_n\big]=0$, for all $n\geq n_0$. The following result is an elementary consequence of the Burkholder-Davis-Gundy and Jensen inequalities.

\begin{lemma}   \label{LemmaBDG}
Let $X$ be an $H$-valued martingale difference with moments of order $p\geq2$. Then
$$
\mathbb E\big[|X_0+\cdots+X_{n-1}|^p\big]^\frac1p \lesssim_p \sqrt n\,\left(\frac1n\Big(\mathbb E\big[\big|X_0\big|^p\big] + \cdots + \mathbb E\big[\big|X_{n-1}\big|^p\big]\Big)\right)^\frac1p. 
$$
In particular, if $X$ is stationary, then
$$
\mathbb E\Big[|X_0+\cdots+X_{n-1}|^p\Big]^\frac1p \lesssim_p \sqrt n\,\|X_0\|_{L^p}. 
$$
\end{lemma}

Assume from now on that we are given a sequence $(X_n)_{\geq n_0}$ of integrable $H$-valued random variables on $(\Omega,\mathcal{F},\mathbb{P})$. For $j\in\mathbb{Z}$, and $k\geq 0$, define the $\sigma$-algebra
$$
\mathcal F^{(k)}_j := \mathcal F_{j2^k},
$$
and set
$$
Y^{(k)}_j := \mathbb E\Big[X_{j2^k}+\cdots+X_{j2^k+(2^k-1)} \big| \mathcal F^{(k)}_{j-1}\Big].
$$
(It may not make sense for all $j, k$, depending on how far in the past the $\sigma$-algebras $(\mathcal F_n)$ are defined.) Note that
$$ 
Y^{(\ell+1)}_j = \mathbb E\Big[Y^{(\ell)}_{2j}+Y^{(\ell)}_{2j+1}\Big|\mathcal F^{(\ell+1)}_{j-1}\Big], 
$$
so 
$$
M^{(\ell)}_j := Y^{(\ell)}_{2j} + Y^{(\ell)}_{2j+1} - Y^{(\ell+1)}_j
$$
is a stationary martingale difference with respect to the filtration $\big(\mathcal F^{(\ell+1)}_j\big)_{j\geq0}$. We use the classical martingale/co-boundary decomposition to prove the next result.

\begin{lemma}   \label{lem.martingaledecomposition}
Fix $p\geq2$, and assume that $\mathcal F_n$ is defined for $n\geq-2^{k+1}$, then
\begin{equation*}\begin{split}
\mathbb E\Big[\big|Y^{(0)}_0 &+ \cdots + Y^{(0)}_{2^k-1}\big|^p\Big]^\frac1p   \\
&\lesssim_p \sum_{0\leq j\leq k} 2^{(k-j)/2}\,\left(\frac1{2^{k-j}}\Big(\mathbb E\big[\big|Y^{(j)}_0\big|^p\big] + \cdots + \mathbb E\big[\big|Y^{(j)}_{2^{k-j}-1}\big|^p\big]\Big)\right)^\frac1p.
\end{split}\end{equation*}
In particular, if the sequence $(X_n)$ is stationary, then
$$
\mathbb E\big[\big|Y^{(0)}_0 + \cdots + Y^{(0)}_{2^k-1}\big|^p\big]^\frac1p \lesssim_p 2^{k/2}\left( \mathbb E\big[\big|Y^{(0)}_0\big|^p\big]^\frac1p  + \cdots + 2^{-k/2}\mathbb E\big[\big|Y^{(k)}_0\big|^p\big]^\frac1p\right).
$$
\end{lemma}

\begin{proof}
For any $0\leq j\leq k$, set $n_j := 2^{k-j}$; note that $n_k=1$. We have for $j<k$ the identity
\begin{align*}
Y^{(j)}_0 + \cdots + Y^{(j)}_{n_j-1}
&= \big(Y^{(j)}_0 + Y^{(j)}_1\big) + \cdots + \big(Y^{(j)}_{2n_{j+1}-2} + Y^{(0)}_{2n_{j+1}-1}\big)   \\
&= M^{(j)}_0+\cdots+M^{(j)}_{n_{j+1}-1} + Y^{(j+1)}_0+\cdots+Y^{(j+1)}_{n_{j+1}-1}. 
\end{align*}
By induction we get
$$
Y^{(0)}_0 + \cdots + Y^{(0)}_{n-1} = \big(M^{(0)}_0+\cdots+M^{(0)}_{n_1-1}\big) + \cdots + \big(M^{(k-1)}_0\big) + Y^{(k)}_0.
$$
Because $M^{(j)}$ is a martingale difference, we know from Lemma \ref{LemmaBDG} that
\begin{equation*}\begin{split}
\mathbb E\big[\big|M^{(j)}_0&+\cdots+M^{(j)}_{n_{j+1}-1}\big|^p\big]^\frac1p   \\
&\lesssim_p \sqrt{n_{j+1}}\cdot\left(\frac1{n_{j+1}}\Big(\mathbb E\big[\big|M^{(j)}_0\big|^p\big] + \cdots + \mathbb E\big[\big|M^{(j)}_{n_{j+1}}\big|^p\big]\Big)\right)^\frac1p.
\end{split}\end{equation*}
We also know that 
$$
M^{(j)}_{2^k} = Y^{(j)}_{2^{k+1}}+Y^{(j)}_{2^{k+1}+1} - \mathbb E\big[Y^{(j)}_{2^{k+1}}+Y^{(j)}_{2^{k+1}+1}\big|\mathcal F^{(j+1)}_{2^k-1}\big],
$$ 
so we have
\begin{align*}
\mathbb E\big[\big|M^{(j)}_{2^k}\big|^p\big]^\frac1p
  & \leq \mathbb E\big[\big|Y^{(j)}_{2^{k+1}}\big|^p\big]^\frac1p
     + \mathbb E\big[\big|Y^{(j)}_{2^{k+1}+1}\big|^p\big]^\frac1p
     + \mathbb E\Big[\mathbb E\big[|Y^{(j)}_{2^{k+1}}|^p\big|\mathcal F^{(j+1)}_{-1}\big]\Big]^\frac1p   \\
  &\quad+ \mathbb E\Big[\mathbb E\big[|Y^{(j)}_{2^{k+1}+1}|^p\big|\mathcal F^{(j+1)}_{-1}\big]\Big]^\frac1p \\
  & \leq 2\mathbb E\big[\big|Y^{(j)}_{2^{k+1}}\big|^p\big]^\frac1p + 2\mathbb E\big[\big|Y^{(j)}_{2^{k+1}+1}\big|^p\big]^\frac1p.
\end{align*}
Putting it all together, we obtain
\begin{align*}
\mathbb E\big[\big|Y^{(0)}_0 &+ \cdots + Y^{(0)}_{2^k-1}\big|^p\big]^\frac1p   \\
  & \lesssim_p \sum_{0\leq j<k}\sqrt{n_{j+1}}\cdot\left(\frac1{2n_{j+1}}\Big(\mathbb E\big[\big|Y^{(j)}_0\big|^p\big] + \cdots + \mathbb E\big[\big|Y^{(j)}_{2n_{j+1}}\big|^p\big]\Big)\right)^\frac1p   \\
  &\quad+ \mathbb E\big[\big|Y^{(k)}_0\big|^p\big]^\frac1p \\
  & \lesssim_p \sum_{0\leq j\leq k} 2^{(k-j)/2}\cdot\left(\frac1{2^{k-j}}\Big(\mathbb E\big[\big|Y^{(j)}_0\big|^p\big] + \cdots + \mathbb E\big[\big|Y^{(j)}_{2^{k-j}-1}\big|^p\big]\Big)\right)^\frac1p.
\end{align*}
\end{proof}

\begin{proof}[Proof of Proposition \ref{lem.kolmogorovwk}]
It is enough to prove that we have for any $T\geq 1$ and $p\geq2$, the estimate
$$
\mathbb E\left[\left|\int_0^Tv_t\,dt\right|^p\right]\lesssim_p T^{p/2}.
$$
Fix the integer $k$ such that $T/2\leq 2^k<T$, and define
$$
X_j := \int_{jT2^{-k}}^{(j+1)T2^{-k}}v_t\,dt,\qquad \mathcal F_j = \sigma\Big(v_s,s\leq(j+1)T2^{-k}\Big).
$$
Since we assume that $v_0$ is distributed according to an invariant probability measure, we can actually have our process started for a time arbitrarily far in the past, so we can assume that $\mathcal F_j$ is well-defined for any $j\geq -2^{k+1}$. We can then write
\begin{align*}
\int_0^Tv_t\,dt & = \big(X_0 - \mathbb E\big[X_0|\mathcal F_{-1}\big]\big) + \cdots + \big(X_{2^k-1} - \mathbb E\big[X_{2^k-1}|\mathcal F_{2^k-2}\big]\big)   \\
      							& \quad+ \mathbb E\big[X_0|\mathcal F_{-1}\big] + \cdots + \mathbb E\big[X_{2^k-1}|\mathcal F_{2^k-2}\big].
\end{align*}
The first sum is a stationary martingale difference with respect to the $\sigma$-algebra $(\mathcal F_j)_{j\geq0}$; the second is the subject of the previous lemma. One then has the estimate
\begin{align*}
\mathbb E&\left[\left|\int_0^Tv_t\,dt\right|^p\right]^\frac1p \lesssim_p 2^{k/2}\,\mathbb E\Big[\big|X_0 - \mathbb E[X_0|\mathcal F_{-1}]\big|^p\Big]^\frac1p \\
  & \quad + 2^{k/2}\left(\mathbb E\Big[\big|Y^{(0)}_0\big|^p\Big]^\frac1p + \cdots + 2^{-k/2}\,\mathbb E\Big[\big|Y^{(k)}_0\big|^p\Big]^\frac1p\right) \\
  & \lesssim_p \sqrt T\left(\mathbb E\Big[\big|X_0\big|^p\Big]^\frac1p + \mathbb E\Big[\big|Y^{(0)}_0\big|^p\Big]^\frac1p + \cdots + 2^{-k/2}\,\mathbb E\Big[\big|Y^{(k)}_0\big|^p\Big]^\frac1p\right)
\end{align*}
with the notations of Lemma \ref{lem.martingaledecomposition}. In our setting,
$$
\|X_0\|_{L^p} = \mathbb E\left[\left|\int_0^{T2^{-k}}v_t\,dt\right|^p\right]^\frac1p \leq \left( T2^{-k}\right)^\frac pp \leq 2 
$$
and
$$
Y^{(j)}_0 = \mathbb E\left[\int_{2^j T2^{-k}}^{2^{j+1}T2^{-k}}v_t\,dt\,\middle|\mathcal F_{-1}\right] = \mathbb E_{v_0}\left[\int_{2^j T2^{-k}}^{2^{j+1}T2^{-k}}v_t\,dt\right].
$$
Note that we have from Corollary \ref{CorExpConvergence}
\begin{equation*}\begin{split}
\left|\mathbb E_{v_0}\left[\int_{2^j T2^{-k}}^{2^{j+1}T2^{-k}}v_t\,dt\right]\right| &\leq \int_{2^j T2^{-k}}^{2^{j+1}T2^{-k}}\big|\mathbb E_{v_0}[v_t]\big|\,dt \lesssim \int_{2^j T2^{-k}}^\infty e^{-t/\tau}\,dt   \\
&\lesssim e^{-2^{j-1}/\tau}.
\end{split}\end{equation*}
We can insert this in the upper bound for the integral to obtain
\begin{equation}   \label{EqTightnessCAlpha}
\left\|\int_0^Tv_t\,dt\right\|_{L^p} \leq \left(1 + \sum_{j\geq0}2^{-j/2}e^{-2^{j-1}/\tau}\right) \sqrt T.
\end{equation}
\end{proof}

\subsubsection{Convergence in H\"older spaces}
\label{sssec.convergenceHolder}

We are ready to prove Proposition \ref{prop.convergenceHolder} on the weak convergence of $X^\sigma$ in any H\"older space $C^\alpha([0,1],H)$ to the Brownian motion in $H$ with covariance given by formula \eqref{EqCovarianceLimitBM}.

\begin{proof}[Proof of Proposition \ref{prop.convergenceHolder}]
From the tightness result in $C^\alpha([0,1],H)$ stated in Proposition \ref{lem.kolmogorovwk}, it is sufficient to show that $X^\sigma$ converges weakly in $C^0([0,1],H)$ to the above mentionned Brownian motion. If suffices for that purpose to see that for a finite sequence of $t_1<\cdots<t_n$, and a small enough positive delay $\epsilon$, the random variables $X^\sigma_{t_i-\epsilon}-X^\sigma_{t_{i-1}}$ converge to finitely many independent Gaussian random variables, with corresponding covariances $(t_i-\epsilon-t_{i-1})$ times the covariance \eqref{EqCovarianceLimitBM}. One can for instance use Dedecker and Merlev\`ede conditional central limit theorem, from Theorem 1 in \cite{DedeckerMerlevede}, to see the convergence to a Gaussian limit with the expected covariance operator. One checks that the four conditions (a)-(d) from Theorem 1 in \cite{DedeckerMerlevede} hold true in our setting. We denote by $\ell\in H^*$ a continuous linear form on $H$.
\begin{itemize}
   \item[(a)] We have from the decorrelation result in Corollary \ref{CorExpConvergence} that
   $$
   \Big|\mathbb{E}\Big[\ell\Big(T^{-1/2}\int_0^T v_tdt\Big)\Big| v_0\Big]\Big| \lesssim T^{-1/2}\int_0^n e^{-t/\tau}dt,   
   $$
   converges to $0$ in $L^1$ as $T$ goes to $\infty$.   \vspace{0.1cm}
   
   \item[(b)] The decorrelation result in Corollary \ref{CorExpConvergence} justifies the use of dominated convergence to justify that
   \begin{equation*}\begin{split}
   \frac{1}{T}\,\mathbb{E}\left[\left(\int_0^T\ell(v_s)ds\right)^2\bigg| \,v_0\right]
   \end{split}\end{equation*}
   converges in $L^1$ as $T$ goes to $\infty$. The limit $\beta^2_\ell$ is constant, from the ergodic behaviour of the velocity process.  \vspace{0.1cm}
   
   \item[(c)] The $L^p$ estimate \eqref{EqTightnessCAlpha} with any $p>1$ shows that the family 
   $$
   \left(\frac{1}{T}\,\int_0^T\ell(v_s)ds\right)_{T\geq 1}
   $$ 
   is uniformly integrable.   \vspace{0.1cm}
   
   \item[(d)] Last we have, by stationarity of the velocity process, that
   \begin{equation*}\begin{split}
   \frac{1}{T}\,\mathbb{E}\left[\left\|\int_0^T v_sds\right\|^2\right] &= \frac{2}{T}\,\iint_{0\leq s\leq t\leq T }\mathbb{E}\big[(v_s,v_t)\big]\,dsdt   \\
   &= 2\int_0^T \mathbb{E}\big[(v_0,v_r)\big]\,dr
   \end{split}\end{equation*}
   converges indeed to a finite limite as $T$ goes to $\infty$. This limit is given by the finite sum $\sum_{i\geq 0} \beta^2_{\ell_i}$, where $(\ell_i)$ stands for an orthonormal basis of $H^*$.
\end{itemize}
One reads the independence of the limit Gaussian random variables corresponding to different time intervals $[t_{i-1},t_i-\epsilon]$ on their null correlation; the latter is a direct consequence of the decorrelation property of Corollary \ref{CorExpConvergence}. The statement of Proposition \ref{prop.convergenceHolder} follows then from the conclusion of Dedecker and Merlev\`ede convergence result. The identification of the covariance \eqref{EqCovarianceLimitBM} is a consequence of the corresponding statement, Proposition 3.4, in the finite dimensional setting of \cite{PierreFiniteDim}.
\end{proof}

We aim now at improving the weak invariance principle of Proposition \ref{prop.convergenceHolder} into a weak invariance principle for the canonical rough path associated with $X^\sigma$. This will be crucial in Section \ref{SectionKBM} when defining kinetic Brownian motion in a diffeomorphism space as the solution of a differential equation driven by $X^\sigma$, and proving the interpolation results of Theorem \ref{ThmHomogenizationM} and Theorem \ref{ThmHomogenizationM0} by a continuity argument. We recall in the next section all we need to know from rough paths theory.

\subsection{The flavor of rough paths theory}
\label{SubsectionRP}

It is not our purpose here to give a detailled account of rough paths theory. We refer the reader to the lecture notes \cite{LyonsStFlour, FrizHairer, BaudoinLNEMS, BailleulFlows}, for introductions to the subject from different point of views. The following will be sufficient for our needs here.

\smallskip

Rough paths theory is a theory of ordinary differential equations 
\begin{equation}\label{EqCODE}
dz_t = \sum_{i=1}^\ell V_i(z_t)\,dh^i_t,
\end{equation}
controlled by non-smooth signals $h\in C^\alpha([0,1],\mathbb{R}^\ell)$. The point $z_t$ moves here in $\mathbb{R}^d$, where we are given sufficiently regular vector fields $V_i$. Young integration theory \cite{Young, LyonsYoung} allows to make sense of the integral $\int_0^\cdot V(y_s)dh_s$, for  paths $y,h$ that are $\alpha$-H\"older, for $\alpha>\frac{1}{2}$, as an $\mathbb{R}^d$-valued $\alpha$-H\"older path depending in locally Lipscthiz way on $y$ and $h$. This allows to formulate the differential equation \eqref{EqCODE} as a fixed point problem for a contracting map from $C^\alpha([0,1],\mathbb{R}^d)$ into itself, and to obtain as a consequence the continuous dependence of the solution path on the driving control $h$. Lyons-Young theory cannot be used for $\alpha$-H\"older controls with $\alpha<\frac{1}{2}$, as even in $\mathbb{R}$, with one dimensional controls, there exists no \textit{continuous} bilinear form on $C^\alpha([0,1],\mathbb{R})\times C^\alpha([0,1],\mathbb{R})$ extending the Riemann integral $\int_0^1 y_tdh_t$, of smooth paths $y,h$; see Propositon 1.29 of \cite{LyonsStFlour}. (This can be understood from a Fourier analysis point of view as a consequence of the fact that the resonant operator from Littlewood-Paley theory is unbounded on $C^\alpha([0,1],\mathbb{R})\times C^{\alpha-1}([0,1],\mathbb{R})$, when $2\alpha-1< 0$; see \cite{BCD}.) Lyons' deep insight was to realize that what really fixes the dynamics of a solution path to the controlled differential equation \eqref{EqCODE} is not only the increments $dh_t$, or $h_t-h_s$, of the control, but rather the increments of $h$ together with the increments of a number of its iterated integrals. This can be understood from the fact that for a smooth control, one has the Taylor-type expansion
\begin{equation*}\begin{split}
f(z_t) = f(z_s) &+ \left(\int_s^t dh^i_u\right)(V_if)(z_s) + \left(\int_{s\leq u_2\leq u_1\leq t} dh^j_{u_2}dh^k_{u_1}\right)\big(V_jV_kf\big)(z_s)    \\
&+ \int_{s\leq u_3\leq u_2\leq u_1\leq t} \big(V_n V_jV_kf\big)(z_{u_3})\,dh^n_{u_3}dh^j_{u_2}dh^k_{u_1},
\end{split}\end{equation*}
for any real-valued smooth function $f$ on $\mathbb{R}^d$. (We use Einstein' summation convention, with integer indices in $[1,\ell]$.) We consider here the vector fields $V_i$ as first order differential operators, so we have for instance
$$
V_jV_kf = (D^2f)(V_j,V_k) + (Df)\big((DV_k)(V_j)\big).
$$
The usual first order Euler scheme
$$
z_t \simeq z_s + (h^i_t-h^i_s)V_i(z_s),
$$
is refined by the above second order Milstein scheme
$$
z_t \simeq z_s + (h^i_t-h^i_s)V_i(z_s) + \left(\int_{s\leq u_2\leq u_1\leq t} dh^j_{u_2}dh^k_{u_1}\right)\big(V_jV_k\big)(z_s),
$$
whose one step error is given explicitly by the above triple integral, of order $\vert t-s\vert^3$, for a $C^1$ control $h$. The iterated integrals 
$$
\int_{s\leq u_2\leq u_1\leq t} dh^j_{u_2}dh^k_{u_1} = \int_{s\leq u_1\leq t} \big(h^j_{u_1}-h^j_s\big)\,dh^k_{u_1},
$$
are however meaningless for a control $h\in C^\alpha([0,1],\mathbb{R}^\ell)$, when $\alpha\leq 1/2$. A $p$-rough path $\bf X$ above $h$, with $2\leq p<3$, is exactly the datum of $h$ together with a quantity, indexed by $(s\leq t)$, that plays the role of these iterated integrals. Set $[0,1]_\leq := \big\{(s,t)\in [0,1]^2\,;\,s\leq t\big\}$, and recall that $\big(\mathbb{R}^\ell\big)^{\otimes 2}$ stands for the set of $\ell\times \ell$ matrices.

\begin{definition}
Fix $2\leq p<3$. A \textbf{$p$-rough path $\bf X$} over $\mathbb{R}^\ell$, is a map
\begin{equation*}\begin{split}
[0,1]_{\leq} &\rightarrow \mathbb{R}^\ell\times\big(\mathbb{R}^\ell\big)^{\otimes 2}   \\
(s,t) &\mapsto \big(X_{ts},\mathbb{X}_{ts}\big),
\end{split}\end{equation*}
such that 
$$
X_{ts} = h_t-h_s,
$$
for a $C^\alpha([0,1],\mathbb{R}^\ell)$ path $h$, and $\mathbb{X}$ satisfies Chen's relations
$$
\mathbb{X}_{ts} = \mathbb{X}_{tu} + X_{us}\otimes X_{tu} + \mathbb{X}_{us},
$$
for all $0\leq s\leq u\leq t\leq 1$. The $1/p$-H\"older norm on $X$, and the $2/p$-H\"older norm on $\mathbb{X}$, define jointly a complete metric on the nonlinear space ${\sf RP}(p)$ of $p$-rough paths.
\end{definition}

Chen's relation accouts for the fact that for a $C^1$ path $h$, one has indeed 
\begin{equation*}\begin{split}
\int_{s\leq u_1\leq t} \big(h^j_{u_1}-h^j_s\big)\,dh^k_{u_1} &= \int_{u\leq u_1\leq t} \big(h^j_{u_1}-h^j_u\big)\,dh^k_{u_1} + \big(h^j_u-h^j_s\big)\big(h^k_t-h^k_u\big)   \\
&\quad+ \int_{s\leq u_1\leq u} \big(h^j_{u_1}-h^j_s\big)\,dh^k_{u_1}
\end{split}\end{equation*}
for any $0\leq s\leq u\leq t\leq 1$, and any indices $1\leq j,k\leq \ell$. One has also in that case, by integration by parts, the identiy
\begin{equation*}\begin{split}
&\int_{s\leq u_1\leq t} \big(h^j_{u_1}-h^j_s\big)\,dh^k_{u_1} + \int_{s\leq u_1\leq t} \big(h^k_{u_1}-h^k_s\big)\,dh^j_{u_1}   \\
&= \frac{1}{2}\, \big(h_t^j-h^j_s\big)\big(h_t^k-h^k_s\big).
\end{split}\end{equation*}
A $p$-rough path $\bfX$ such that the symmetric part of $\mathbb{X}_{ts}$ is equal to $\frac{1}{2}\,X_{ts}\otimes X_{ts}$, for all times $0\leq s\leq t\leq 1$, is called \textbf{weakly geometric}. The set of weakly geometric $p$-rough paths is closed in ${\sf RP}(p)$. For a $C^1$ path $h$ defined on the time interval $[0,1]$, setting $X_{ts} := h_t-h_s$ and 
$$
\mathbb{X}_{ts} := \int_s^t X_{us}\otimes dX_u,
$$
for all $0\leq s\leq t\leq 1$, defines a weak geometric $p$-rough path, for any $2\leq p<3$, called the \textbf{canonical rough path associated with $h$.} Let $B$ stand for an $\ell$-dimensional Brownian motion. The Stratonovich Brownian rough path ${\bf B} = (B,\mathbb{B})$ is defined by 
$$
\mathbb{B}_{ts} := \int_{s\leq u\leq t} (B_u-B_s)\otimes {\circ d}B_u.
$$
It is almost surely a weak geometric $p$-rough path, for any $2<p<3$.

\begin{definition}   \label{DefnRDESolution}
Let $C^3_b$ vector fields $(V_i)_{1\leq i\leq \ell}$ on $\mathbb{R}^d$ be given, together with a weak geometric $p$-rough path $\bfX$ over $\mathbb{R}^\ell$. A path $(z_t)_{0\leq t\leq 1}$ is said to be a solution to the rough differential equation 
\begin{equation}\label{EqRDE}
dz_t = V(z_t)\,d{\bfX}_t
\end{equation}
if there is an exponent $a>1$, such that one has
\begin{equation}\begin{split}   \label{EqDefnRDESolution}
f(z_t) = f(z_s) &+ X_{ts}^i(V_if)(z_s) + \mathbb{X}^{jk}_{ts}\big(V_jV_kf\big)(z_s) + O\big(\vert t-s\vert^a\big),
\end{split}\end{equation}
for any smooth real-valued function $f$ on $\mathbb{R}^d$, and any times $0\leq s\leq t\leq 1$. 
\end{definition}

The above $O(\cdot)$ term is allowed to depend on $f$. Importantly, the solution of a rough differential equation driven by the Stratonovich Brownian rough path coincides almost surely with the solution of the corresponding Stratonovich differential equation; see e.g. the lecture notes \cite{FrizHairer, BailleulLN}.

\begin{theorem}[Lyons' universal limit theorem]   \label{ThmRDESolution}
The rough differential equation \eqref{EqRDE} has a unique solution. It is an element of $C^{1/p}([0,1],\mathbb{R}^d)$ that depends continuously on $\bfX$.
\end{theorem}

The map that associates to the driving rough path the solution to a given rough differential equation, seen as an element of $C^{1/p}([0,1],\mathbb{R}^d)$, is called the \textbf{It\^o-Lyons solution map}. If $({\bfX}^n)$ is a sequence of random geometric $p$-rough path in $\mathbb{R}^\ell$, converging weakly to a limit random geometric $p$-rough path $\bfX$, the continuity of the It\^o-Lyons solution map gives for free the weak convergence in $C^{1/p}([0,1],\mathbb{R}^d)$ of the laws of the solutions to Equation \eqref{EqRDE} driven by the ${\bfX}^n$, to the law of the solution of that equation driven by $\bfX$. 

The theory works perfectly well for dynamics with values in Banach spaces or Banach manifolds, and driving rough paths ${\bfX}=(X,\mathbb{X})$, with $X$ taking values in a Banach space $E$. One needs to take care in that setting to the tensor norm used to define the completion of the algebraic tensor space $E\otimes_a E$, as this may produce non-equivalent norms, and that norm is used to define the norm of a rough path. Note that families of vector fields $(V_1,\dots, V_\ell)$ are then replaced in that setting by one forms on $E$ with values in the space of vector fields on the space where the dynamics takes place. See e.g. Lyons' original work \cite{Lyons98} or Cass and Weidner's work \cite{CassWeidner} for the details. See e.g. \cite{BailleulSeminaire} for a simple proof of Lyons' universal limit theorem in that general setting.

The vector fields in Definition \ref{DefnRDESolution} and Theorem \ref{ThmRDESolution} are required to be $C^3_b$. This is used to get solution of equation \eqref{EqRDE} that are defined on the whole time interval $[0,1]$. Only local in time existence results can be obtained when working with unbounded vector fields, or on a manifold. The Taylor-like expansion property \eqref{EqDefnRDESolution} defining a solution path is then only required to hold for each time $s$, for $t$ sufficiently close to $s$. One still has continuity of the solution path with respect to the driving rough path, in an adapted sense. See e.g. Section 2.4.2 of \cite{ABT}. This continuity property is sufficient to obtain the local weak convergence of the laws of the solution path to the corresponding limit path, for random driving weak geometric $p$-rough paths converging weakly to a limit random weak geometric $p$-rough path. See Definition \ref{DefnLocalWeakConvergence} for the definition of local weak convergence.

\smallskip

So far, we have defined kinetic Brownian motion $(x^\sigma_t,v^\sigma_t)$ in $H$ from its unit velocity process $v^\sigma$. We have seen in Proposition \ref{prop.convergenceHolder} that its time rescaled position process $(X^\sigma_t) := (x^\sigma_{\sigma^2t})$ is converging weakly in $C^\alpha\big([0,1],H\big)$ to a Brownian motion with explicit covariance \eqref{EqCovarianceLimitBM}, for any $\alpha<1/2$. We prove in the next section that the canonical rough path ${\bfX}^\sigma$ associated with $X^\sigma$ converges weakly as a weak geometric $p$-rough path to the Stratonovich Brownian rough path associated with the Brownian motion with covariance \eqref{EqCovarianceLimitBM}, for any $2<p<3$. This convergence result will be instrumental in Section \ref{SectionKBM} to prove that the Cartan development in diffeomorphism spaces of the time rescaled kinetic Brownian motion in Hilbert spaces of vector fields converge to some limit dynamics as $\sigma$ increases to $\infty$. This will come as a direct consequence of the continuity of the It\^o-Lyons solution map.

\begin{remark}
The idea of using rough paths theory for proving elementary homogenization results was first tested in the work \emph{\cite{FrizGassiatLyons}} of Friz, Gassiat and Lyons, in their study of the so-called \emph{physical Brownian motion in a magnetic field.} That random process is described as a $C^1$ path $(x_t)_{0\leq t\leq 1}$ in $\R^d$ modeling the motion of an object of mass $m$, with momentum $p = m\dot x$, subject to a damping force and a magnetic field. Its momentum satisfies a stochastic differential equation of Ornstein-Uhlenbeck form
$$
dp_t = -\frac{1}{m}\,Mp_tdt + dB_t,
$$
for some matrix $M$, whose eigenvalues all have positive real parts, and $B$ is a $d$-dimensional Brownian motion. While the process $(Mx_t)_{0\leq t \leq 1}$ is easily seen to converge to a Brownian motion $W$, its rough path lift is shown to converge in a rough paths sense in $L^q$, for any $q\geq 2$, to a random rough path \emph{different from} the Stratonovich Brownian rough path associated to $W$. 

A number of works have followed this approach to homogenization problems for fast-slow systems; see \cite{ABT, KellyMelbourne1, KellyMelbourne2, BailleulCatellier, CFKMZh} for a sample.
\end{remark}

\subsection{Rough paths invariance principle for the canonical lift}
\label{SubsectionCvgceRP}

As in Section \ref{SubsectionCvgceX}, \textit{we assume in all of this section that the initial condition $v_0$ of the velocity process of kinetic Brownian motion is distribued according to its invariant probability measure $\mu$, from Theorem \ref{PropInvariantMeasure}.} 

\smallskip

Let ${\mathbf X}^\sigma = (X^\sigma,\mathbb{X}^\sigma)$ stand for the canonical rough path associated to the random $\mathcal C^1$ path $X^\sigma$, where we recall that
$$
\mathbb X^\sigma_{ts} = \int_s^t(X^\sigma_u-X^\sigma_s)\otimes dX^\sigma_u  = \frac1{\sigma^4}\int_{\sigma^4s}^{\sigma^4t}\int_{\sigma^4s}^uv_r\otimes v_u\,drdu.
$$
Recall that the tensor space $H\otimes H$ is equipped with its natural complete Hilbert(-Schmidt) norm.

\subsubsection{Tightness in rough paths space}
\label{sssec.tighnessRP}

\begin{proposition}
\label{lem.kolmogorovRP}
For any $p\geq2$, we have
$$
\sup_{\sigma>0}\mathbb E\big[|\mathbb X^\sigma_{t,s}|^p\big]\lesssim|t-s|^p.
$$
\end{proposition}

It follows in particular from Proposition \ref{lem.kolmogorovwk}, Lemma \ref{lem.kolmogorovRP} and the known Kolmogorov-Lamperti criterion for rough paths that the family of laws $\mathcal L({\bf X}^\sigma)$ is tight in $\mathrm{RP}(\alpha^{-1})$, for any $1/3<\alpha<1/2$.

\begin{proof}
The statement of the lemma is a consequence of the estimate
$$ 
\mathbb E\left[\bigg|\int_0^T\int_0^t v_s\otimes v_t\,dsdt\bigg|^p\right]\lesssim_p T^p,
$$
for $T\geq 1$; we prove the latter. We use for that purpose the same kind of multiscale martingale/coboundary decomposition as in the proof of Lemma \ref{lem.martingaledecomposition}. Let $k$ the unique integer such that 
$$
1\leq\delta:=T2^{-k}<2.
$$ 
Define
$$
A_j := \int_{j\delta}^{(j+1)\delta}\int_0^tv_s\otimes v_t\,dsdt,
$$
and 
$$
\widehat{\mathcal F}_j := \mathcal F_{(j+1)\delta}=\sigma\Big(v_s,s\leq(j+1)\delta\Big).
$$
As above, we can assume without loss of generality that $\widehat{\mathcal F}_j$ is defined for all $j\geq-2^{k+1}$, as $v_0$ is assumed to be distributed according to the invariant probability measure of the velocity process. Then the integral rewrites as
\begin{align}
\label{eq.doubleintdec}
\begin{split}
\int_0^T\int_0^tv_s\otimes v_t&\,dsdt   \\
    &= \Big(A_0 - \mathbb E\big[A_0|\widehat{\mathcal F}_{-1}\big]\Big) + \cdots + \Big(A_{2^k-1} - \mathbb E\big[A_{2^k-1}|\widehat{\mathcal F}_{2^k-2}\big]\Big)   \\
   & \quad + \mathbb E\big[A_0|\widehat{\mathcal F}_{-1}\big] + \cdots + \mathbb E\big[A_{2^k-1}|\widehat{\mathcal F}_{2^k-2}\big]
\end{split}
\end{align}
The first sum is a martingale difference with respect to $(\widehat{\mathcal F}_n)_{n\geq0}$, albeit not stationary,
\begin{equation*}\begin{split}
\mathbb E\bigg[\Big|\sum_{0\leq j<2^k}\Big(A_j - &\mathbb E\big[A_j|\widehat{\mathcal F}_{j-1}\big]\Big)\Big|^p\bigg]^\frac1p   \\
&\lesssim_p 2^{k/2} \,\bigg(2^{-k}\sum_{0\leq j<2^k}\mathbb E\Big[\big|A_j - \mathbb E\big[A_j|\widehat{\mathcal F}_{j-1}\big]\big|^p\Big]\bigg)^\frac1p \\
& \lesssim_p 2^{k/2} \,\bigg(2^{-k}\sum_{0\leq j<2^k}\mathbb E\big[\big|A_j\big|^p\big]\bigg)^\frac1p.
\end{split}\end{equation*}
Each term is controlled using Lemma \ref{lem.kolmogorovwk}, and the fact that $|v_t|=1$,
$$
\mathbb E\big[\big|A_j\big|^p\big] \leq \delta^{p-1}\int_{j\delta}^{(j+1)\delta}\mathbb E\Big[\Big|\int_0^tv_s\,ds\Big|^p\Big]\,dt
\lesssim_p \int_{j\delta}^{(j+1)\delta}t^{p/2}\,dt \lesssim (2^k)^{p/2},
$$
so the $L^p$ norm of the first sum in \eqref{eq.doubleintdec} is bounded above by $2^k$, up to a constant depending only on $p$.

\smallskip

The second sum in \ref{eq.doubleintdec} is treated as in the proof of Lemma \ref{lem.martingaledecomposition}. Set  here
$$
Z^{(n)}_j := \mathbb{E}\Big[A_{j2^n}+\cdots+A_{j2^n+(2^n-1)}\Big|\widehat{\mathcal{F}}^{(n)}_{j-1}\Big],
$$
with 
$$
\widehat{\mathcal{F}}^{(n)}_j := \widehat{\mathcal{F}}_{(j-1)2^n}.
$$
One has 
\begin{equation*}\begin{split}
\mathbb E\bigg[\bigg|\sum_{0\leq j<2^k}&\mathbb E\big[A_j\big|\widehat{\mathcal F}_{j-1}]\bigg|^p\bigg]^\frac1p   \\
&\lesssim_p \sum_{0\leq n\leq k} 2^{(k-n)/2} \left(\frac1{2^{k-n}}\Big(\mathbb E\big[\big|Z^{(n)}_0\big|^p\big] + \cdots + \mathbb E\big[\big|Z^{(n)}_{2^{k-n}-1}\big|^p\big]\Big)\right)^\frac1p,
\end{split}\end{equation*}
and we are left with the study of the moments of the $Z^{(n)}_j$. These variables are the conditional expectation of a double integral, which can be decomposed at time $(j-1)2^n\delta+\delta$ as follows.
\begin{equation*}\begin{split}
Z^{(n)}_j & = \mathbb E\left[\int_{j2^n\delta}^{(j+1)2^n\delta}\int_0^tv_s\otimes v_t\,dsdt\ \Big|\widehat{\mathcal F}_{(j-1)2^n} \right]   \\
& = \int_{j2^\ell\delta}^{(j+1)2^n\delta}\int_0^{(j-1)2^n\delta+\delta\vee0}v_s\otimes\mathbb E\Big[v_t\,\big|\,\widehat{\mathcal F}_{(j-1)2^n}\Big]\,dsdt  \\
&\quad  + \int_{j2^n\delta}^{(j+1)2^n\delta}\mathbb E\left[\int_{(j-1)2^n\delta+\delta\vee0}^tv_s\otimes\mathbb E\big[v_t|\mathcal F_s\big]\,ds\,\Big|\,\widehat{\mathcal F}_{(j-1)2^n}\right]\,dt   \\
& =: R^{(n)}_j + S^{(n)}_j.
\end{split}\end{equation*}
Because the conditioning is from a distant past, the first term is controlled using the exponential mixing and the estimate of Lemma \ref{lem.kolmogorovwk}.
\begin{align*}
\mathbb E\Big[\big|R^{(n)}_j\big|^p\Big]
& = \mathbb E\left[\left|\int_0^{(j-1)2^n\delta+\delta\vee0}v_s\,ds\right|^p\,\left|\int_{j2^n\delta}^{(j+1)2^n\delta}\mathbb E\big[v_t\,\big|\,\widehat{\mathcal F}_{(j-1)2^n}\big]\,dt\right|^p\right]   \\
& \lesssim \mathbb E\left[\left|\int_0^{(j-1)2^n\delta+\delta\vee0}v_s\,ds\right|^p\right]\,\left(\int_{2^n\delta}^{2^{n+1}\delta} e^{-(t-\delta)/\tau}\,dt\right)^p \\
& \lesssim_p (2^{k-n})^\frac p2(2^n)^\frac p2 e^{-p2^n/\tau}
\end{align*}
When dealing with the second term, we use the stationarity of $v$ to write
\begin{align*}
\big|S^{(n)}_j\big|
& \leq \int_{j2^n\delta}^{(j+1)2^n\delta}\mathbb E\left[\int_{(j-1)2^n\delta+\delta}^t\Big|v_s\otimes\mathbb E\big[v_t|\mathcal F_s\big]\Big|\,ds\,\Big|\,\widehat{\mathcal F}_{(j-1)2^n}\right]\,dt \\ & \stackrel{\mathcal L}
  = \int_{2^n\delta}^{2^{n+1}\delta}\mathbb E\left[\int_\delta^t\Big|v_s\otimes\mathbb E\big[v_t|\mathcal F_s\big]\Big|\,ds\,\Big|\,\widehat{\mathcal F}_0\right]\,dt   \\
& \lesssim \int_{2^n\delta}^{2^{n+1}\delta}\mathbb E\Big[\int_\delta^t e^{-(t-s)/\tau}\,ds\,\Big|\,\widehat{\mathcal F}_0\Big]\,dt \\
& \lesssim 2^n.
\end{align*}
Now we have, for each $0\leq n\leq k$ and $0\leq j<2^{k-n}$,
$$
\mathbb E\big[\big|Z^{(\ell)}_j\big|^p\big] \lesssim_p (2^{k-\ell})^\frac p2(2^\ell)^\frac p2\cdot e^{-p2^\ell/\tau} + 2^{p\ell}
$$
so we eventually have
\begin{align*}
\mathbb E\left[\left|\sum_{0\leq j<n}\mathbb E\big[A_j | \widehat{\mathcal F}_{j-1}\big]\right|^p\right]^\frac1p
& \lesssim_p \sum_{0\leq\ell\leq k}\big( 2^{k-\ell}2^{\ell/2} e^{-2^\ell/\tau} + 2^{(k-\ell)/2}2^\ell \big)   \\
& = 2^k\sum_{0\leq\ell\leq k}\big( 2^{-\ell/2}\, e^{-2^\ell/\tau} + 2^{-(k-\ell)/2} \big)   \\
& = 2^k\sum_{0\leq\ell\leq k}2^{-\ell/2}\big(1+ e^{-2^\ell/\tau}\big).
\end{align*}
This last sum is convergent, so the $L^p$ norm of the second term in \eqref{eq.doubleintdec} is no greater than a constant multiple of $2^k$.
\end{proof}

\subsubsection{Convergence in rough path space}
\label{sssec.convergenceRP}

We are now ready to state and prove the main result of this section.

\begin{theorem}
\label{ThmConvergenceRP}
Pick $1/3<\alpha<1/2$. The processes $\bfX^\sigma$ converge in law in ${\sf RP}(\alpha^{-1})$, as $\sigma$ goes to $\infty$, to the Stratonovich Brownian rough path with covariance
$$ 
C(\ell,\ell') = \int_0^\infty\mathbb E\Big[\ell(v_0)\ell'(v_t)+\ell'(v_0)\ell(v_t)\Big]\,dt
$$
\end{theorem}

Let $\bf X$ be a random weak geometric $\alpha^{-1}$-rough path with distribution an arbitrary limit point of the  family of laws of the $\bfX^\sigma$. Write ${\bf X} = (B,\mathbb{X})$, with $B$ a Brownian motion with the above covariance. Denote by $\underline{\bf X}$ the projection of $\bf X$ on the finite dimensional space generated by the first $d$ vectors of the basis $(e_i)$ from Section \ref{SubsectionBMSphere} -- we use below the associated coordinate system. Using a monotone class argument and the tightness result stated in Lemma \ref{lem.kolmogorovRP}, the statement of Theorem \ref{ThmConvergenceRP} is a consequence of the following result, given that $d\geq 1$ is arbitrary.

\begin{lemma}
The $d$-dimensional random rough path $\underline{\bf X}$ is a Stratonovich Brownian rough path with associated covariance matrix $\mathrm{diag}(\gamma_1,\cdots,\gamma_d)$, with
$$
\gamma_i := 2\int_0^\infty\mathbb E\big[v_0^iv_t^i\big]\,dt.
$$
\end{lemma}

\begin{proof}
Let $G_d^{2}$ stand for the step-$2$ nilpotent Lie group over $\mathbb{R}^d$. We prove that the process $(\underline{{\bf X}}_{t0})_{0\leq t\leq 1}$ is a $G_d^{2}$-valued Brownian motion by showing that it has stationary, independent, increments. The stationarity is inherited from the stationarity of the ${\bf X}^\sigma$. The independence of the increments of $\underline{\bf X}$ on disjoint closed intervals is a consequence of Corollary \ref{CorMixingCoefficient} on the convergence to $0$ of the mixing coefficient of $(v_t)$. Continuity of $\underline{\bf X}$ allows to extend the result to adjacent time intervals.

\medskip

We identify the generator of the $G_d^2$-valued Brownian motion $(\underline{\bf X}_t)$ as the generator of the $d$-dimensional Stratonovich Brownian rough path following the method of \cite{PierreFiniteDim}. We recall the details for the reader's convenience. Note that we only need to consider the joint dynamics of $\underline{B}_t$ and the antisymmetric part $(\underline{\mathbb{A}}_t)$ of $(\underline{\mathbb{X}}_t)$; the former takes values in the Lie algebra $\frak{g}_d^2$ of $G_d^2$ -- a linear space. Denote by $\underline{\mathbb{A}}^B$ the antisymmetric part of Stratonovich Brownian rough path associated with $\underline{B}$. We then have, for any smooth real-valued function $f$ on $\mathbb{R}^d\times\frak{g}_d^2$ with compact support, the identity
\begin{equation*}\begin{split}
&\Big(f\big(\underline{B}_t, \underline{\mathbb{A}}_t\big) - f(0)\Big) - \Big(f\big(\underline{B}_t, \underline{\mathbb{A}}^B_t\big) - f(0)\Big)   \\
&= (\partial_2f)(\underline{B}_t, 0)\big(\underline{\mathbb{A}}_t-\underline{\mathbb{A}}^B_t\big) + O\Big(\big|\underline{\mathbb{A}}_t-\underline{\mathbb{A}}^B_t\big|^2\Big)   \\
&= \Big((\partial_2f)(\underline{B}_t, 0) - (\partial_2f)(0, 0) \Big)\big(\underline{\mathbb{A}}_t-\underline{\mathbb{A}}^B_t\big) +  (\partial_2f)(0, 0)\big(\underline{\mathbb{A}}_t-\underline{\mathbb{A}}^B_t\big)   \\
&\quad+ O\Big(\big|\underline{\mathbb{A}}_t-\underline{\mathbb{A}}^B_t\big|^2\Big).
\end{split}\end{equation*}
The conclusion follows by multiplying by $t^{-1}$ and taking expectation, sending $t$ to $0$, after recalling that $\underline{\mathbb{A}}_t$ and $\underline{\mathbb{A}}^B_t$ are centered, and recalling the uniform estimates from Proposition \ref{lem.kolmogorovRP} under the form
$$
\big\|\underline{\mathbb{A}}_t\big\|_{L^2} \vee \big\|\underline{\mathbb{A}}^B_t \big\|_{L^2} \lesssim t.
$$
\end{proof}

\section{Geometry of the configuration space}
\label{SectionGeometry}

\subsection{Configuration space}
\label{SubsectionConfigurationSpace}

Let $(M,g)$ be a $d$-dimensional connected and oriented Riemannian manifold, and $\pi : F\rightarrow M$ a finite dimensional fiber bundle over $M$, with vertical bundle $VF\rightarrow M$. Think of the trivial bundles $M\times M\rightarrow M$, or $M\times TM\rightarrow M$, as typical examples. We collect from Palais' seminal work \cite{Palais} elementary results on the Hilbert manifold $H^s(F)$ of sections of $\pi$ with Sobolev regularity exponent $s>\frac{d}{2}$. 

\begin{enumerate}
   \item \textbf{Sobolev embedings} hold true, with in particular $H^s(F)\subset C^k(M,F)$, if $s>k+\frac{d}{2}$ and $k\geq 0$.   \vspace{0.1cm}
   
   \item {\bf Variations of $H^s$-sections of $F$.} The spaces $TH^s(F)$ and $H^s(VF)$ are isomorphic as Hilbert manifolds. This isomorphism accounts for the fact that an infinitesimal perturbation $(\delta f)$ of a section $f$ of $F$, reads as a collection of vertical tangent vectors $(\delta f)(x) \in V_{f(x)}F$, indexed by $x\in M$. As a particular example, for any finite dimensional manifold $N$, the spaces $TH^s(M,N)$ and $H^s(M,TN)$ are isomorphic.   \vspace{0.1cm}
   
   \item For any two finite dimensional fiber bundles $F,G$ above $M$, the map 
   $$
   (f,g) \mapsto \big(x\in M \mapsto (f(x),g(x))\big)   
   $$
   is an isomorphism between $H^s(F)\times H^s(G)$ and $H^s(F\times_M G)$.   \vspace{0.1cm}

   \item \textbf{Omega lemma.} Given a smooth fiber bundle morphism $\Phi : F\rightarrow G$, above $M$, set 
   $$
   \omega_\Phi(f) := \Phi\circ f, 
   $$
   for any section $f$ of $F$. Then $\omega_\Phi$ sends $H^s(F)$ in $H^s(G)$, and $d\omega_\Phi : TH^s(F)\rightarrow TH^s(G)$ is isomorphic to $\omega_{d\Phi} : H^s(VF)\rightarrow H^s(VG)$, via the isomorphisms $TH^s(F)\simeq H^s(VF)$ and $TH^s(F')\simeq H^s(VG)$.      
\end{enumerate}

\medskip

For $s>\frac{d}{2}$, set 
$$
\mathscr{M} := H^s(M,M);
$$
this will be the \textbf{configuration space} of our dynamics. Choosing $s>\frac{d}{2}$, ensures that $\mathscr{M}\subset C^0(M,M)$, by Sobolev embedings. The tangent space to this Hilbert manifold is given by
$$
T\mathscr{M} \simeq H^s(M,TM),
$$
from item 2 above. If $s>\frac{d}{2}+1$, elements of $\mathscr{M}$ are $C^1$ maps from $M$ into itself. Recall in that case from Section 4 of \cite{EbinMarsden} that the subset $\mathscr{M}_0$ of $\mathscr{M}$ of $H^s$ maps from $M$ into itself that preserve the volume form by pull-back is then a closed submanifold of $\mathscr{M}$, and that elements of $\mathscr{M}_0$ are diffeomorphisms. So $\mathscr{M}_0$ is a group. We shall always assume implicitly these constraints on the regularity exponent $s$, when talking about $\mathscr{M}$ or $\mathscr{M}_0$. We recall other elementary facts on $H^s(TM)$ at the end of this section.

To implement a version of Cartan's development machinery in the weak Riemannian setting of the next section, we introduce the following finite dimensional fiber bundles above $M$, seen below as the first component. Given $x,y \in M$, denote by ${\sf O}(T_xM,T_yM)$ the set of isometries from $T_xM$ to $T_yM$. Set
\begin{equation*}\begin{split}
&F^{(e)} := \Big\{(x,y;e)\,;\,(x,y)\in M\times M, e\in {\sf O}(T_xM,T_yM)\Big\},   \\
&F^{(w)} := \Big\{(x,y;w)\,;\,(x,y)\in M\times M, w\in T_xM\Big\},   \\
&F^{(v)} := \Big\{(x,y;v)\,;\,(x,y)\in M\times M, v\in T_yM\Big\},   \\
&F^{(e,v)} := \Big\{\big(x,y;e,v\big)\,;\,(x,y)\in M\times M,  e\in {\sf O}(T_xM,T_yM), v\in T_yM\Big\}.
\end{split}\end{equation*}   \vspace{-0.4cm}

\begin{figure}[ht]\par
\scalebox{.6}{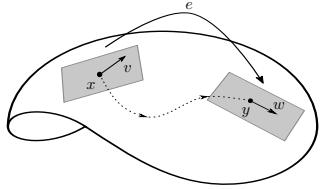}
\vspace{-0.3cm}
\caption{An infinitesimal rigid object $x$ is moving along a path. It has position $y$ and velocity $v$ at some time. Its orientation at that time is given by an isometry $e : T_xM\rightarrow T_yM$, and its velocity $v$ is given in its initial reference frame by $w$.}
\label{fig.fibrebundles}
\end{figure}

We understand $H^s(F^{(v)})$ as the set of $H^s$ maps from $M$ into $TM$, so $T\mathscr{M}\simeq H^s\big(F^{(v)}\big)$. We denote by $\big(\varphi(\cdot),v(\cdot)\big)$ a generic element of $H^s(F^{(v)})$. We have similar interpretations of the other $H^s$ spaces over the corresponding bundles, with similar notations. Since the map
\begin{equation*}\begin{split}
F^{(e)}\times_{M\times M}F^{(w)} &\rightarrow F^{(v)}   \\
\big((x,y ; e),(x,y;w)\big) &\mapsto \big(x,y;e(w)\big),
\end{split}\end{equation*}
is a smooth bundle morphism, it follows from items 3 and 4 above, that it induces a \textit{smooth map} from $H^s(F^{(w,e)})$ into $H^s(F^{(v)})$. Similarly, the smooth map 
\begin{equation*}\begin{split}
F^{(e,v)} &\rightarrow F^{(w)}  \\
\big(x,y ; e,v\big) &\mapsto \big(x,y;e^{-1}(v)\big),
\end{split}\end{equation*}
induces a \textit{smooth map} from $H^s(F^{(e,v)})$ into $H^s(F^{(w)})$.

\medskip

We refer the reader to the classic textbook \cite{Rosenberg} for the following elementary facts from functional analysis about the Laplace operator $\Delta$ on vector fields on $M$. We take the convention that $-\Delta$ is a non-positive symmetric operator on $L^2(TM)$. This operator has compact resolvant, so one has an eigenspaces decomposition
\begin{equation}\label{EqEigenspaceDecomposition}
L^2(TM) = \bigoplus_{n\geq 0} E_{\lambda_n},
\end{equation}
with finite dimensional eigenspaces $E_{\lambda_n}$, with corresponding non-positive eigenvalues $\lambda_n\downarrow -\infty$. Eigenvectors of $-\Delta$ are smooth, from elliptic regularity results. We recover the space $H^s(TM)$ described above setting
$$
H^s(TM) = \left\{f=\sum_{n\geq 0}f_n \in L^2(T\M)\,;\,\sum_{n\geq 0}\lambda_n^s \|f_n\|_{L^2}^2<\infty\right\}.
$$
The $0$-eigenspace is finite dimensional. Any choice of Euclidean norm $\|\cdot\|$ on it defines the topology of $H^s(TM)$, associated with the norm
$$
\|f\|_s := \|f_0\| + \left(\sum_{n\geq 0}\lambda_n^s \|f_n\|_{L^2}^2\right)^{1/2}.
$$

\bigskip

\subsection{Weak Riemannian structure on the configuration space}
\label{SubsectionRiemannianStructure}

Denote by $\textsc{Vol}$ the Riemannian volume measure on $(M,g)$, and by $\exp : TM\rightarrow M$, its exponential map. The configuration space $\mathscr{M}$ is endowed with a smooth weak Riemannian structure, setting for any $\varphi\in \mathscr{M}$ and $X(\varphi),Y(\varphi)\in T_\varphi\mathscr{M}$,
\begin{equation}\label{EqL2Metric}
\big(X(\varphi),Y(\varphi)\big)_\varphi := \int_M g_{\varphi(m)}\big(X(\varphi)(m),Y(\varphi)(m)\big)\,\textsc{Vol}(dm).
\end{equation}
This formula defines by restriction a weak Riemannian metric on the space $\mathscr{M}_0$ of $H^s$ maps from $M$ into itself preserving the volume form. In that setting, notice that if $X(\varphi) = \textbf{X}\circ\varphi$ and $Y(\varphi) = \textbf{Y}\circ\varphi$, for some vector fields $\textbf{X},\textbf{Y}$ on $M$, then the change of variable formula gives 
$$
\big(X(\varphi),Y(\varphi)\big)_\varphi = \int_M g_m\big(\textbf{X}(m),\textbf{Y}(m)\big)\,\textsc{Vol}(dm),
$$
so the scalar product is in that case the $L^2$ scalar product of the vector fields $\textbf{X}$ and $\textbf{Y}$. The fact that the topology on $\mathscr{M}$ induced by the scalar product is weaker than the $H^s$-topology makes non-obvious the existence of a smooth Levi-Civita connection. Ebin and Marsden have proved that 
\begin{itemize}
   \item the $L^2$ metric \eqref{EqL2Metric} is a smooth function on $\mathscr{M}$,   \vspace{0.1cm}
   
   \item it has a smooth Levi-Civita connection $\overline\nabla$, with associated exponential map $\textrm{Exp}$ well-defined and smooth in a neighbourhood of the zero section; it is explicitly given by 
   $$
   \textrm{Exp}_\varphi(X)(m) = \textrm{exp}_{\varphi(m)}\big(X(m)\big).
   $$
\end{itemize}
The geodesics of $(\mathscr{M},\overline{\nabla})$ are defined for all times. Denote by $\nabla$ the Levi-Civita connection of $(M,g)$. For smooth right invariant vector fields $X, Y$ on $\mathscr{M}$, with $X(\varphi) = \textbf{X}\circ\varphi$ and $Y(\varphi) = \textbf{Y}\circ\varphi$, one has
$$
(\overline{\nabla}_XY)(\varphi) = (\nabla_{\textbf{X}}\textbf{Y})\circ\varphi.
$$
The $L^2$-scalar product is right invariant on the group $\mathscr{M}_0$, from the change of variable formula. The Levi-Civita connection of the $L^2$ metric on the volume preserving configuration space $\mathscr{M}_0$ is explicitly given in terms of the Hodge projection operator $P$ on divergence-free vector fields on $M$. Denote by $R_\varphi$ the right composition by $\varphi$. For any $\varphi\in\mathscr{M}_0$, the map 
\begin{equation}\label{EqHodgeProjector}
P_\varphi := dR_\varphi\circ P\circ dR_\varphi^{-1},
\end{equation}
is indeed the orthogonal projection map from $T_\varphi\mathscr{M}$ into $T_\varphi\mathscr{M}_0$, and its depends smoothly on $\varphi\in\mathscr{M}_0$. So the Levi-Civita connection $\overline{\nabla}^0$ on $\mathscr{M}_0$ is given by 
$$
\overline{\nabla}^0 = P\circ \overline{\nabla};
$$
it is a smooth map. Its associated exponential map is no longer given by the exponential map on $TM$, due to the non-local volume preserving constraint. Geodesics are not defined for all times anymore. Denote by Id the identity map on $M$. For smooth right invariant vector fields $X, Y$ on $\mathscr{M}$, with $X(\varphi) = \textbf{X}\circ\varphi$ and $Y(\varphi) = \textbf{Y}\circ\varphi$, for vector fields $\textbf{X},\textbf{Y}$ on $M$, one has
$$
\big(\overline{\nabla}_X^0Y\big)(\textrm{Id}) = P\big(\nabla_{\textbf{X}}\textbf{Y}\big).
$$
V.I. Arnol'd showed formally in his seminal work \cite{Arnold} that the velocity field $u : [0,T]\rightarrow H^s(TM)$ of a geodesic $\varphi_t$ in $\mathscr{M}_0$, with $u_t := \dot \varphi_t\circ\varphi^{-1}_t$, is a solution to Euler's equation for the hydrodynamics of an incompressible fluid. Ebin and Marsden gave an analytical proof of that fact in their seminal work \cite{EbinMarsden}. (Besides that classical reference, we refere the reader to Arnold and Khesin's book \cite{ArnoldKhesin}, or Smolentsev's thourough review \cite{Smolentsev} for reference works on the weak Riemannian geometry of the configuration space.)

\smallskip

The flat two-dimensional torus ${\bf T}^2$ offers an interesting concrete example. Its symplectic structure allows to identify a Hilbert basis $(A_k,B_k)_{k\in\mathbb{Z}^2\backslash{0}}$ of $T_{\textrm{Id}}\mathscr{M}_0$ from an eigenbasis for the Laplace operator on real-valued functions on $\textbf{T}^2$; see e.g. Arnold and Khesin's book \cite{ArnoldKhesin}, Section 7 of Chap. 1. Denote by $\partial_1, \partial_2$ the constant vector fields in the coordinate directions, and $k=(k_1,k_2)\in\mathbb{Z}^2$. One has 
\begin{equation*}\begin{split}
A_k &= \vert k\vert^{-1}\Big(k_2\cos(k\cdot\theta)\partial_1 - k_1\cos(k\cdot\theta)\partial_2\Big),   \\
B_k &= \vert k\vert^{-1}\Big(k_2\sin(k\cdot\theta)\partial_1 - k_1\sin(k\cdot\theta)\partial_2\Big).
\end{split}\end{equation*}
One can see in the following simulations the image of axis circles by the time $1$ map of the associated flow in ${\bf T}^2$, corresponding to different inital conditions for $u_0$, with $\varphi_0=\textrm{Id}$. The simulations were done using an elementary finite dimensional approximation for the dynamics, using the explicit expressions for the Christoffel symbols first given by Arnold in \cite{Arnold}. 
\begin{figure}[ht]
\centering
\begin{center}\begin{minipage}{12cm}
\begin{minipage}{4cm}\includegraphics[scale=0.2]{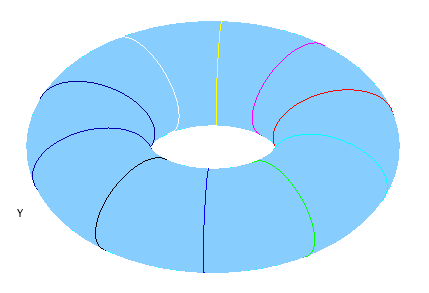}\end{minipage}\hspace{-0.1cm}
\begin{minipage}{4cm}\includegraphics[scale=0.2]{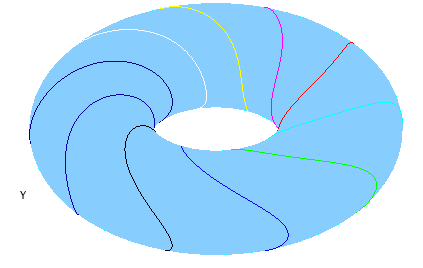}\end{minipage}\hspace{-0.3cm}
\begin{minipage}{4cm}\includegraphics[scale=0.2]{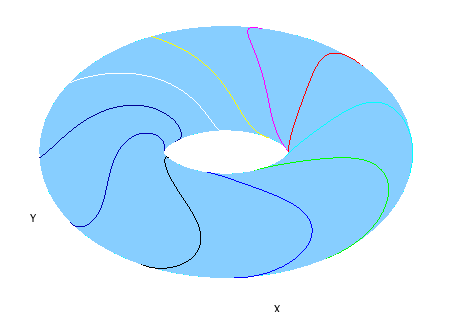}\end{minipage}
\end{minipage}\end{center}
\caption{Time $1$ snapshots of the geodesic flow, for different initial momenta in the volume preserving diffeomorphism group.}
\end{figure}
We come back to this point in Section \ref{SubsectionCartan}.

\bigskip

\subsection{Parallel transport}
\label{SubsectionParallelTransport}

We recast in this section the parallel transport operations in $\mathscr{M}$ and $\mathscr{M}_0$, using the bundles $F$ from Section \ref{SubsectionConfigurationSpace}. This allows to set the notations for the next section on Cartan development operation in $\mathscr{M}$ and $\mathscr{M}_0$. Recall $H^s(F)$ stands for $H^s$ sections from $M$ into the corresponding bundle $F$. We denote by $VF$ the vertical space in $TF$, for the canonical projection map $F\rightarrow M$. Recall also that $T_\textrm{Id}\mathscr{M}$ is simply the set of $H^s$ vector fields on $M$.

\smallskip

Denote by $K : TTM \rightarrow TM$, the connector associated with the Levi-Civita connection $\nabla$ on $M$. So, for a path $\gamma_t = (m_t,v_t)$ in $TM$, one has 
$$
\nabla_{\dot m_t}v_t = K(\dot \gamma_t),
$$
and 
$$
\nabla_\textbf{X}\textbf{Y} = K\big((d\textbf{Y})(\textbf{X})\big),
$$
for any smooth vector fields $\textbf{X},\textbf{Y}$ on $M$. The second order tangent bundle $TT\mathscr{M}$ of $\mathscr{M}$ identifies with $H^s(M, TTM)$. The connector $\overline{K}$ associated with the $L^2$-Levi-Civita connection $\overline{\nabla}$ is given, for a section $Y$ of $TTM$ over an element of $\mathscr{M}$, by 
$$
\overline{K}(Y) := K\circ Y \in T\mathscr{M}.
$$
Denote by $V_2F^{(v)}$ the vertical space in $TF^{(v)}$ for the canonical projection map 
$$
p_2 : F^{(v)}\rightarrow M\times M.
$$ 
One defines a smooth one form on $V_2F^{(v)}$, with values in $TF^{(v)}$, by requiring that $\nabla_{\dot y_t}v_t=0$ iff 
$$
\frac{d}{dt}\,(y_t,v_t) = \frak{H}^{(v)}(y_t,v_t;\dot y_t).
$$
We choose the letter $\frak{H}$, for this horizontal lift of the connection. In simple terms, for any fixed $(y,v)\in TM$, the linear map $\frak{H}^{(v)}(y,v;\cdot)$ identifies the space $T_yM$ to the horizontal subspace of $T_{(y,v)}TM$, via the usual horizontal lift. Note that the definition of $\frak{H}^{(v)}(y,v ; \dot y)$ does not depend on the base point $x\in M$, for a generic element $(x,y ; v)\in F^{(v)}$ and $\dot y\in T_yM$. 

Denote also by $ \frak{H}^{(e)}$ the smooth one form on $V_2F^{(v)}$ with values in the space of vector field on $F^{(e)}$, such that for any path $(x,y_t;e_t)$ in $F^{(e)}$, and any vector $w\in T_xM$, the vector $e_t(w)\in T_{y_t}M$ is transported parallely along the $M$-valued path $(y_t)$ iff
$$
\frac{d}{dt}\,(y_t,e_t) =  \frak{H}^{(e)}(y_t,e_t;\dot y_t).
$$
Here again, the base point $x\in M$ is not involved in the definition of the tangent vector $\frak{H}^{(e)}(y,e;\dot y)$, for a generic element $(x,y ; e)\in F^{(e)}$ and $\dot y\in T_yM$. Pick 
$$
(x_0,y_0;e_0)\in F^{(e)},
$$
and note that for any vertical vector 
$$
(\dot y,\dot e)\in V_{(x_0,y_0 ; e_0)} F^{(e)},
$$ 
and $v_0\in T_{y_0}M$, one has
$$
(\dot y,\dot e) =  \frak{H}^{(e)}\big(y_0,e_0 ; v_0\big)
$$
iff 
$$
\big(\dot y, \dot e(w)\big) = \frak{H}^{(v)}\big(y_0,e_0(w);v_0\big) \in V_{(x_0,y_0; e_0(w))}F^{(v)},
$$
for any $w\in T_yM$, with $\dot e(w)$ defined naturally. It follows from the Omega Lemma that one defines a \textit{smooth vector field} $\overline{\frak{H}}^{(v)}$ on $H^s(F^{(v)})$, setting
$$
\overline{\frak{H}}^{(v)}\big(\varphi(\cdot),v(\cdot)\big) := \frak{H}^{(v)}\circ\big(\varphi(\cdot),v(\cdot);v(\cdot)\big).
$$
(Note that while $\frak{H}^{(v)}$ is a one form with values in vector fields, $\overline{\frak{H}}^{(v)}$ is indeed a vector field.) Similarly, we define a \textit{smooth one-form} on $T_\textrm{Id}\mathscr{M}$ with values in vector fields on $H^s(F^{(e)})$, setting
$$
\overline{\frak{H}}^e\big(\varphi(\cdot),e(\cdot) ; {\bf X}\big) :=  \frak{H}^{(e)}\circ\big(\varphi(\cdot),e(\cdot); e({\bf X})\big), \qquad {\bf X}\in T_\textrm{Id}\mathscr{M}.
$$

\begin{proposition}   \label{PropCartanM0Argument}
Given a path $\big(\varphi_t(\cdot); e_t(\cdot),v_t(\cdot)\big)_{0\leq t\leq 1}$ in $H^s(F^{(e,v)})$, one has pointwise
$$
\frac{d}{dt}\,\big(\varphi_t(x),e_t(x)\big) =  \frak{H}^{(e)}\big(\varphi_t(x),e_t(x);v_t(x)\big),
$$
for all $x\in M$, iff
$$
\frac{d}{dt}\,\big(\varphi_t,e_t({\bf X})\big) = \overline{\frak{H}}^{(v)}\big(\varphi_t,e_t({\bf X});v_t\big),
$$
for every ${\bf X}\in T_\textrm{\emph{Id}}\mathscr{M}$.
\end{proposition}

The next two propositions give a description of parallel transport in $\mathscr{M}$ and $\mathscr{M}_0$, respectively, in terms of the vector field $\overline{\frak{H}}^{(v)}$ on $H^s(F^{(v)})$.

\begin{proposition}
\label{PropParallelTransportM}
Let $\big(\varphi_t(\cdot),v_t(\cdot)\big)_{0\leq t\leq 1}$ be a $T\mathscr{M}$-valued path. Then 
$$
\overline{\nabla}_{\dot \varphi_t}v_t = 0,
$$
iff
$$
\frac{d}{dt}\big(\varphi_t,v_t\big) = \overline{\frak{H}}^{(v)}\big(\varphi_t, v_t;\dot \varphi_t\big).
$$
\end{proposition}

\begin{proof}
Given $(y,v)\in TM$, the following map identifies $T_yM$ with the vertical subspace of $T_{(y,v)}TM$
$$
\frak{V}^{(v)}(y,v;\cdot) : w\in T_yM\mapsto\frac{d}{dt}_{\big| t=0} (v+tw)\in T_{(y,v)}TM.
$$
For any $(x,y;v)\in F^{(v)}$ and $u\in T_{(y,v)}\big(F^{(v)}_x\big)$, one then has
$$
u = \frak{H}^{(v)}(y,v ; a) + \frak{V}^{(v)}(y,v ; b) \quad \textrm{ iff } \quad a = dp_2(u), \textrm{ and } b = K(u).
$$
For an $H^s(F_v)$-valued path $\big(\varphi_t(\cdot),v_t(\cdot)\big)$, one then has the splitting
\begin{equation}\label{EqDerivativeSplitting}
\begin{split}
\frac{d}{dt}\,(\varphi_t,v_t)
& = \frak{V}^{(v)}\circ\big(\varphi_t,v_t;K\circ\dot v_t)\big) + \frak{H}^{(v)}\circ\big(\varphi_t,v_t;\dot\varphi_t\big) \\
& = \frak{V}^{(v)}\circ\big(\varphi_t,v_t;\Nabla_{\dot\varphi_t}v_t)\big) + \overline{\frak{H}}^{(v)}\circ\big(\varphi_t,v_t;\dot\varphi_t\big).  
\end{split}
\end{equation}
The result follows because composition by $\frak V_v(y,v;\cdot)$ is one-to-one.
\end{proof}

Recall that $P$ stands for Hodge projector on divergence-free vector fields.

\begin{proposition}
\label{PropParallelTransportM0}
Let $\big(\varphi_t(\cdot),v_t(\cdot)\big)_{0\leq t\leq 1}$ be a $T\mathscr{M}_0$-valued path. Then
$$
\overline{\nabla}^0_{\dot \varphi_t}v_t = 0,
$$
iff
$$
\frac{d}{dt}\big(\varphi_t,v_t\big) = (dP)\Big(\overline{\frak{H}}^{(v)}\big(\varphi_t,v_t,;\dot \varphi_t\big)\Big).
$$
\end{proposition}

\begin{proof}
Write $T_{\mathscr{M}_0}\mathscr{M}$ for the section of $T\mathscr{M}$ above $\mathscr{M}_0$, and write $Q := \Id-P : T_{\mathscr{M}_0}\mathscr{M}\to T_{\mathscr{M}_0}\mathscr{M}$, for the projection on the orthogonal in $T\mathscr{M}$ of $T\mathscr{M}_0$. Note that the differential $dP$ of $P$ identifies to $P$ in the fibers, since it is linear. The identification is up to an isomorphism which is exactly the composition by $\frak V_v$, in the sense that
$$
dP\big(\frak{V}^{(v)}(\varphi,v ; v')\big) = \frak{V}^{(v)}\circ\big(\varphi,v ; P(v')\big)
$$
for any $v,v'\in T_\varphi\mathscr M$. As we work with a $T\mathscr{M}_0$-valued path $(\varphi_t,v_t)$, one has $Q(v_t) = 0$, at all times, so differentiating this identity with respect to $t$ gives 
$$
dQ(\dot v_t)=0. 
$$
Since $P+Q=\Id$, we can conclude with the decomposition \eqref{EqDerivativeSplitting}, by rewriting the expression for the time derivative under the form
\begin{align*}
\frac{dv_t}{dt}
& = dP(\dot v_t) + dQ(\dot v_t)   \\
& = \frak{V}^{(v)}\circ\Big(\varphi_t,v_t ; P\big(K(\dot v_t)\big)\Big) + dP\big(\frak{H}^{(v)}\circ\big(\varphi_t,v_t ; \dot\varphi_t\big)\big)   \\
& = \frak{V}^{(v)}\circ\Big(\varphi_t,v_t ; \overline\nabla^0_{\dot\varphi_t}v_t\Big) + dP\Big(\overline{\frak{H}}^{(v)}\circ(\varphi_t,v_t;\dot\varphi_t)\Big). \qedhere
\end{align*}
\end{proof}

\bigskip

\subsection{Cartan and Lie developments}
\label{SubsectionCartan}

Cartan's moving frame method \cite{Cartan} provides a mechanics for constructing $C^1$ paths on $M$ from $C^1$ path on $\mathbb{R}^d$, giving something of a chart on pathspace in $M$. Its description requires the introduction of the orthonormal frame bundle $OM$ over $M$. It is made up of pairs $z=(m,e)$, with $m\in M$ and $e$ an isometry from $\mathbb{R}^d$ to $T_mM$. It has a natural finite dimensional manifold structure, and the Riemannian connection on $TM$ induces vector fields $H_1,\dots, H_d$ on $OM$ by parallel transport of a frame in the direction of its $i^\textrm{th}$ direction along the corresponding path in $M$. 
\begin{figure}[ht]
\centering
\def\svgwidth{6cm}
\includegraphics[scale=0.4]{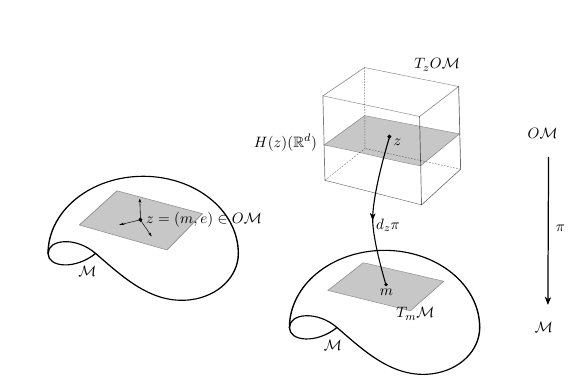}
\vspace{-0.3cm}
\caption{For $z\in OM$ and $a=(a_1,\dots,a_d)\in\mathbb{R}^d$, we have $H(z)(a) := \sum_{i=1}^d a_iH_i(z)\in T_zOM$.}
\end{figure} 
The development in $M$ of a path $(x_t)_{0\leq t\leq 1}$ in $\mathbb{R}^d$ is the natural projection $(m_t)$ in $M$ of the $OM$-valued path $(z_t)$ solution to the equation
$$
\dot z_t = H(z_t)(\dot x_t).
$$
Explosion may happen before time $1$. This path in $M$ depends not only on $m_0$ but also on $e_0$. Conversely, given any $C^1$ path $(m_t)_{0\leq t\leq 1}$ in $M$ and $z_0 = (m_0,e_0) \in OM$ above $m_0$, parallel transport of $e_0$ along the path $(m_t)_{0\leq t\leq 1}$ defines a path $(z_t)_{0\leq t\leq 1}$ in $OM$, and setting $x_t := \int_0^t e_s^{-1}(\dot m_s)\,ds$, defines a path in $\mathbb{R}^d$ whose Cartan development is $(m_t)_{0\leq t\leq 1}$. Geodesics are Cartan's development of straight lines in $\mathbb{R}^d$.

\smallskip

We recast the definition of Cartan development given above in a finite dimensional setting in the following form well suited for the present infinite dimensional setting.

\begin{definition}
\label{DefnCartanDevM}
Let a $C^1$ path $({\bf X}_t)$ in $T_\textrm{\emph{Id}}\mathscr{M}$ be given. An $\mathscr{M}$-valued path $(\varphi_t)$ is the \textbf{Cartan development of $({\bf X}_t)$} if there exists a family 
$$
e_t : T_\textrm{\emph{Id}}\mathscr{M}\to T_{\varphi_t}\mathscr{M},
$$ 
of bounded linear maps, with $e_0=\Id$, such that
\begin{equation}\begin{split}
\label{EqCartanM}
&\dot \varphi_t = e_t(\dot {\bf X}_t), \\
&\overline{\nabla}_{\dot \varphi_t} e_t({\bf Y}) = 0, \quad \text{ for all }{\bf Y}\in T_\textrm{\emph{Id}}\mathscr{M},
\end{split}\end{equation}
at all times where $\varphi_t$ is well-defined.
\end{definition}

This definition conveys the same picture as above. The map $e_t$, named `frame', is transported parallely along the path $(\varphi_t)$, while $\dot \varphi_t$ is given by the image by $e_t$ of $\dot {\bf X}_t$. The existence of a unique Cartan development for a path $({\bf X}_t)$ in $T_\textrm{Id}\mathscr{M}$ is elementary in that case. It follows from Proposition \ref{PropParallelTransportM} that equation \eqref{EqCartanM} is equivalent to requiring that the $H^s\big(F^{(e)}\big)$-valued path $(\varphi_t,e_t)$ satisfies the equation
\begin{equation}   \label{EqCartanODEM}
\frac{d}{dt}(\varphi_t,e_t) = \overline{\frak{H}}^e\big(\varphi_t,e_t; \dot {\bf X}_t\big).
\end{equation}
Since the one-form $\overline{\frak{H}}^e$ is smooth, this equation has a unique solution until its possibly finite explosion time.

\smallskip

Here is now the form of Cartan development dynamics in $\mathscr{M}_0$. Recall $T_\textrm{Id}\mathscr{M}_0$ is the set of $H^s$ divergence-free vector fields on $M$.

\begin{definition}
\label{DefnCartanDevM0}
Let a $C^1$ path $({\bf X}_t)$ in $T_\textrm{\emph{Id}}\mathscr{M}_0$ be given. An $\mathscr{M}_0$-valued path $(\varphi_t)$ is the \textbf{Cartan development of $({\bf X}_t)$} if there exists a family 
$$
e_t : T_\textrm{\emph{Id}}\mathscr{M}_0\to T_{\varphi_t}\mathscr{M}_0,
$$ 
of bounded linear maps, with $e_0=\Id$, such that
\begin{equation}\label{EqParallelTransportM0}
\begin{split}
&\dot \varphi_t = e_t(\dot {\bf X}_t), \\
&\overline{\nabla}^0_{\dot \varphi_t} e_t({\bf Y}) = 0, \quad \text{ for all }{\bf Y}\in T_\textrm{\emph{Id}}\mathscr{M}_0,
\end{split}\end{equation}
at all times where $\varphi_t$ is well-defined.
\end{definition}

The proof of existence of a unique solution to Cartan's development system \eqref{EqParallelTransportM0} in $\mathscr{M}_0$ is not fundamentally different from the case of $\mathscr{M}$, and uses Proposition \ref{PropParallelTransportM0} instead of Proposition \ref{PropParallelTransportM}. It is however more technical, and full details are given in Appendix \ref{AppendixCartanM0}. The system is recast as a controlled ordinary differential equation in the state space 
$$
\mathscr{Z} := H^s(F^{(e)})\times {\sf L}\big(H^s(TM)\big),
$$
with generic element $\big((\varphi,e),f\big)$, and dynamics of the form
\begin{align*}
\frac{d}{dt}\,(\varphi_t,e_t) &= \overline{\frak{H}}^e\Big(\varphi_t,e_t ; f_t(\dot {\bf X}_t)\Big),   \\
\frac{d}{dt}\, f_t & = \overline{\frak{H}}^f\left(\frac{d}{dt}\,(\varphi_t,e_t), f_t\right),
\end{align*}
driven by a \textit{smooth} vector field-valued one form on $T_\textrm{Id}\mathscr{M}_0$. We use Cartan's development map in the configuration manifolds $\mathscr{M}$ and $\mathscr{M}_0$ in the next section. We conclude this section by a brief comparison between Cartan development and the Lie group notion of development, commonly used to define the stochastic Euler equation.

\medskip

Let $G$ stand for a finite dimensional Lie group with Lie algebra $\textsc{Lie}(G)$. Lie's development operation provides another way of constructing paths 
$$
(g_t)_{0\leq t\leq 1}
$$ 
with values in $G$ from paths $(x_t)_{0\leq t\leq 1}$ in $\mathbb{R}^d$, by identifying $T_{g_0}G$ and $\mathbb{R}^d$ via a linear map $\iota_0$, and solving the ordinary differential equation
$$
\dot g_t = \iota_0(\dot x_t)\,g_t.
$$
In such a group setting, Malliavin and Airault \cite{MalliavinAirault} gave a correspondance between the Cartan and Lie notions of development, although this was certainly known to practitioners before; see also \cite{CFM}. Choose an orthonormal basis of the Lie algebra of $G$, and denote by $c_{k,\ell}^n$ the structure constants, so the Christoffel symbols are given by $\Gamma_{k,\ell}^n = \frac{1}{2}\,\big(c_{k,\ell}^n-c_{\ell,n}^k+c_{n,k}^\ell\big)$. Write $\Gamma_k$ for the antisymmetric endomorphism with matrix $\Gamma_{k,\cdot}^\cdot$ in the chosen basis, for $1\leq k\leq d$, and consider $\Gamma$ as a linear map from $\mathbb{R}^d$ into the set of antisymmetric endomorphism of the Lie algebra. Denote by $O\textsc{Lie}(G)$ the orthonormal group of $\textsc{Lie}(G)$.

\begin{proposition}   \label{PropMalliavin}
Let $(w_t)_{0\leq t\leq 1}$ be a $C^1$ path in the Lie algebra of $G$. The path $(g_t)_{0\leq t\leq 1}$ solution to the  $\big(O\textsc{Lie}(G)\times G\big)$-valued equation
\begin{equation}\begin{split}\label{EqLieCartan}
dO_t :=& \;O_t\,\Gamma(\dot w_t)\,dt, \quad O_0=\textrm{\emph{Id}},  \\
dg_t \vspace{0.1cm}=&\;O_t(\dot w_t) g_t,
\end{split}\end{equation}
is the Cartan development of the path $(w_t)$.
\end{proposition}

(The system \eqref{EqLieCartan} is reminiscent of the equation in 
$$
H^s(F^{(e)})\times {\sf L}\big(H^s(TM)\big)
$$ 
from Appendix \ref{AppendixCartanM0}, recasting Cartan's development dynamics in $\mathscr{M}_0$.)  The geodesic started from the identity of $G$, with direction $\omega\in \textsc{Lie}(G)$, is in particular given in the Lie picture as the solution $(g_t)_{0\leq t\leq 1}$ to the equation 
$$
\dot g_t = \exp\big(t\Gamma(\omega)\big)(\omega)\,g_t.
$$
Note that $\exp\big(t\Gamma(\omega)\big)(\omega)\in\textsc{Lie}(G)$. Note also that it is the fact that the Christoffel symbols are constants that allows to reduce the second order differential equation for the geodesics on a generic Riemannian manifold into a first order differential equation, in a Riemannian Lie group setting.

\smallskip

Following Euler's picture, it is this group-oriented point of view that has been considered so far in the geometric viewpoint on fluid hydrodynamics, deterministic or stochastic. The naive implementation of Cartan's machinery in terms of Lie development runs into trouble in the infinite dimensional setting of $\mathscr{M}$ or $\mathscr{M}_0$. This can be seen on the example of the two dimensional torus and the volume preserving diffeomorphism group as a consequence of the fact that Christoffel symbols define antisymmetric unbounded operators that have no good exponential in the orthonormal group of $T_{\textrm{Id}}\mathscr{M}_0$. The problem comes from the fact that $\mathscr{M}$ of $\mathscr{M}_0$ have a \textit{fixed} regularity. See Malliavin's works \cite{Malliavin, CFM} for a quantification of the loss of regularity of Brownian motion in the set of homeomorphisms of the circle, as time increases. The Lie development picture of Cartan's development map can however be used for numerical purposes for simulating kinetic Brownian motion in $\mathscr{M}_0$. It corresponds to having $\dot w_t$ a Brownian motion on the unit sphere of the $H^s$ space of divergence-free vector fields on $M$; see Section \ref{SectionKBM}.

\section{Kinetic Brownian motion on the diffeomorphism group}
\label{SectionKBM}

Pick $s>\frac{d}{2}$, or $s>\frac{d}{2}+1$, depending on whether we work on $\mathscr{M}$ or $\mathscr{M}_0$.

\subsection{Kinetic Brownian motion in $\mathscr{M}$}
\label{SubsectionKBMM}

Set $H := H^s(TM)$. Pick another exponent $a>\frac{1}{2}$, and let $\mathcal{H}$ stand for the $L^2$-orthogonal of $\textrm{ker}(\Delta)$ in $H^{s+a}(TM)$, with norm
$$
\|f\|_{s+a}^2 = \sum_{n\geq 1} \vert\lambda_n\vert^{s+a}\|f_n\|_{L^2}^2,
$$
inherited from the eigenspace decomposition \eqref{EqEigenspaceDecomposition} of $L^2(TM)$. Let $\iota$ stand for the continuous inclusion of $\mathcal{H}$ into $H$. The continuous symmetric operator $\iota\iota^* : H\rightarrow H$, is trace-class, as a consequence of Weyl's law on a closed manifold, so it is the covariance of an $H$-valued Brownian motion $W$. Note the correspondance $\overline{C}=\iota\iota^*$, and
$$
\alpha_n^2 = \vert\lambda_n\vert^{-a},
$$
with the notations of Section \ref{SubsectionBMSphere}. \textit{We assume that the trace condition}
\begin{equation}   \label{EqTraceCondition}
3\alpha_1^2 < \textrm{tr}(\overline{C}), 
\end{equation}
\textit{holds true}. Note that the faster $\lambda_i$ goes to $\infty$, the lesser there is noise in $W$. The extreme case corresponds to only finitely many non-null $\alpha_i$. On the other extreme, the bigger the multiplicity of $\alpha_1^2$ is, the more noise there is in $W$. The trace condition \eqref{EqTraceCondition} holds automatically as soon as $\alpha_1^2$ has multiplicity three.

\smallskip

The Brownian motion $v^\sigma_t$ on the sphere $S$ of $H$, associated with the injection $\mathcal{H}\hookrightarrow H$, is defined as the solution to the stochastic differential equation
$$
dv^\sigma_t = \sigma\,P_{v^\sigma_t}({\circ dW_t}),
$$ 
where $P_a : H\rightarrow H$, is the orthogonal projection on $\langle a\rangle^\perp$, for any $a\neq 0$, and the position process $x^\sigma_t$ of kinetic Brownian motion $\big(x^\sigma_t, v^\sigma_t\big)$ in $H$, given as its integral
$$
x^\sigma_t = x_0 + \int_0^t v^\sigma_s\,ds.
$$
Kinetic Brownian motion on $\mathscr{M}$ is then defined as Cartan development in $\mathscr{M}$ of the time rescaled kinetic Brownian motion $\big(x^\sigma_{\sigma^2t}\big)$ in $H$.

\begin{definition}
Kinetic Brownian motion on $\mathscr{M}$ is the projection $\varphi^\sigma_t$ on the configuration space $\mathscr{M}$ of the solution $\big(\varphi^\sigma_t, e^\sigma_t\big)$ to the equation in $H^s(F^{(e)})$
\begin{equation}   \label{EqDevptM}
\frac{d}{dt}\,(\varphi^\sigma_t,e^\sigma_t) = \overline{\frak{H}}^e\Big(\varphi^\sigma_t,e^\sigma_t; \sigma^2v_{\sigma^2t}^\sigma\Big),
\end{equation}
with initial condition $\varphi_0=\textrm{\emph{Id}}$ and $e_0 = \textrm{Id}\in {\sf L}\big(H^s(TM)\big)$.
\end{definition}

This equation is only locally well-posed. We introduce the following definition to deal with weak convergence questions for possibly exploding solutions of random or stochastic differential equations. Add a cemetary point $\dag$ to $H^s(F^{(e)})$, and endow the disjoint union $H^s(F^{(e)})\sqcup\{\dag\}$ with its natural topology. Denote by $\Omega_0$ the set of continuous paths $z : [0,1]\rightarrow H^s(F^{(e)})\sqcup\{\dag\}$, that start from a reference point $z_0 := (\textrm{Id}, e_0)$ above the identity map on $M$, and that stay at the cemetery point $\dag$, if it leaves $H^s(F^{(e)})$. Let $\mathcal{F}:= \bigvee_{t \in [0,1]} \mathcal{F}_t$ where $(\mathcal{F}_t)_{0\leq t\leq 1}$ stands for the filtration generated by the canonical coordinate process on pathspace. Let $B_R$ stand for the $H^s$ balls with center $z_0$  and radius $R$, for any $R>0$. The first exit time from $B_R$ is denoted by $\tau_R$, and used to define a measurable map 
$$
T_R : \Omega_0\rightarrow C\big([0,1], \overline{B}_R\big),
$$
which associates to any path $(z_t)_{0 \leq t \leq 1} \in \Omega_0$ the path which coincides with $z$ on the time interval $\big[0,\tau_R\big]$, and which is constant, equal to $z_{\tau_R}$, on the time interval $\big[\tau_R,1\big]$. The following definition then provides a convenient setting for dealing with sequences of random process whose limit may explode.

\begin{definition}   \label{DefnLocalWeakConvergence}
A sequence $(\mathbb{Q}_n)_{n\geq 0}$ of probability measures on $\big(\Omega_0,\mathcal{F}\big)$ is said to \textbf{converge locally weakly} to some limit probability $\mathbb{Q}$ if the sequence $\mathbb{Q}_n\circ {T}_R^{-1}$  of probability measures on $C([0,1], \overline{B}_R)$ converges weakly to $\mathbb{Q}\circ {T}_R^{-1}$, for every $R>0$.
\end{definition}

We proved in Theorem \ref{ThmConvergenceRP} that the canonical rough path lift ${\bf X}^\sigma$ of  $\big(x^\sigma_{\sigma^2t}\big)_{0\leq t\leq 1}$, converges weakly in the space of weak geometric $p$-rough paths in $H$, to the Stratonovich Brownian rough path ${\bf B} = (B,\mathbb{B})$, with covariance operator
$$ 
C_B(\ell,\ell') = \int_0^\infty\mathbb E\Big[\ell(v_0)\ell'(v_t)+\ell'(v_0)\ell(v_t)\Big]dt, \quad \ell,\ell'\in H^*.
$$
Since one can rewrite Equation \eqref{EqDevptM} as a rough differential equation driven by the rough path ${\bf X}^\sigma$
$$
\frac{d}{dt}\,(\varphi^\sigma_t,e^\sigma_t) = \overline{\frak{H}}^e\Big(\varphi^\sigma_t,e^\sigma_t; d{\bf X}^\sigma_t\Big),
$$
the continuity of the It\^o-Lyons solution map gives the following theorem. Recall that the solution of a rough differential equation driven by the Stratonovich Brownian rough path coincides almost surely with the solution of the corresponding Stratonovich differential equation.

\begin{theorem}   \label{ThmHomogenizationM}
The $\mathscr{M}$-valued part $(\varphi^\sigma_t)$ of kinetic Brownian motion is converging locally weakly to the projection on $\mathscr{M}$ of the $H^s(F^{(e)})$-valued Brownian motion $(\varphi_t, e_t)$ solution to the stochastic differential equation
$$
\frac{d}{dt}\,(\varphi_t, e_t) = \overline{\frak{H}}^e\Big((\varphi_t, e_t) ; {\circ d}B_t\Big).
$$
\end{theorem}

The motion of $\varphi_t$ itself is not given as the solution of a stochastic differential equation. This happens already in finite dimension, when defining anisotropic Brownian motion on a $d$-dimensional Riemannian manifold $M$ as Cartan development of an anisotropic Brownian motion in $\mathbb{R}^d$. One needs the moving orthonormal frame attached to the running point on $M$, to define the position increment in $M$ from the increment of the driving anisotropic Brownian motion in $\mathbb{R}^d$. The motion in $M$ is in particular non-Markovian, while the motion in $OM$ is Markovian. The same phenomenon happens in the present infinite dimensional setting, and we do not get here classical semimartingale flows in $H^s(M,M)$ \cite{Kunita}, or Brownian flows in critical spaces, such as in Malliavin's work on the canonical Brownian motion on the diffeomorphism group of the circle \cite{Malliavin, Fang, AiraultRen}.

We remark here that the stochastic homogenization methods that X.-M. Li used in \cite{XueMei1} to prove the homogenization result for kinetic Brownian motion in a finite dimensional, complete, Riemannian manifold, require a positive injectivity radius and a uniform control on the gradient of the distance function over the whole manifold. It is unclear that anything like that is available in the present infinite dimensional setting, or in the setting of volume-preserving diffeomorphisms investigated in the next section, especially given the fact that $\mathscr{M}$ or $\mathscr{M}_0$ have infinite negative curvature in some directions. The robust pathwise approach of rough paths allows to circumvent these potential issues.

\subsection{Kinetic Brownian motion in $\mathscr{M}_0$}
\label{SubsectionKBMM0}

Let $H_0$ stand for the closed subspace of $H$ of divergence-free vector fields on the fluid domain $M$. It is the tangent space at the identity map of the closed submanifold $\mathscr{M}_0$ of $\mathscr{M}$ of diffeomorphisms that leave invariant the Riemannian volume form of $M$. The intersection $\mathcal{H}^{s+a}_0$ of $\mathcal{H}^{s+a}$ with $H_0$, is continuously embedded into $H_0$. If $\iota_0$ stands for this injection, the continuous symmetric operator $\iota_0\iota_0^* : H_0\rightarrow H_0$, is trace-class, so it is the covariance of an $H_0$-valued Brownian motion $W$. The spectrum of $\overline{C}_0 := \iota_0\iota_0^*$ is explicit in the example of the $2$-dimensional torus, with maximal eigenvalue $1$, with multiplicity $4$. The trace condition \eqref{EqTraceCondition} thus holds true for any $a>\frac{1}{2}$, in that case. Similarly, the spectrum of the Laplacian operator on vector fields on the $2$-dimensional sphere is obtained from the spectrum of the Laplacian operator on real-valued functions on the $2$-sphere, as a consequence of its canonical symplectic structure \cite{ArakelyanSavvidy, Yoshida}. Eigenvectors are constant multiples of the complex spherical harmonics, so eigenvalues have multiplicity at least two. Here as well, symmetry properties of the $2$-dimensional sphere imply that they have actually multiplicity four, so the trace condition \eqref{EqTraceCondition} holds for free. More generally, divergence-free vector fields on a simply connected $d$-dimensional manifold $M$ are gradients of functions, so one gets the spectrum of the covariance operator $C$ from the spectrum of the Laplacian operator on real-valued functions on $M$. One needs to assume the trace condition \eqref{EqTraceCondition} in this generality.

\smallskip

Kinetic Brownian motion $(x^\sigma_t, v^\sigma_t)$ in $H_0$ is defined as above from the associated Brownian motion $(v^\sigma_t)$ on the sphere $S_0$ of $H_0$, and its integral. We prove in Theorem \ref{ThmCartanM0} of Appendix \ref{AppendixCartanM0} that the Cartan development $\varphi^\sigma_t$ in $\mathscr{M}_0$, of the time rescaled kinetic Brownian motion in $H_0$ is the $\mathscr{M}_0$-part of the solution $(\varphi^\sigma_t, e^\sigma_t, f^\sigma_t)$, to a controlled ordinary differential equation on 
$$
\mathscr{Z} = H^s(F^{(e)})\times {\sf L}\big(H^s(TM)\big)
$$ 
driven by a \textit{smooth} vector field
\begin{align*}
\frac{d}{dt}\,\big(\varphi^\sigma_t,e^\sigma_t\big) &= \overline{\frak{H}}^e\Big(\varphi^\sigma_t,e^\sigma_t ; f^\sigma_t\big(\sigma^2 v^\sigma_t\big)\Big),   \\
\frac{d}{dt}\, f^\sigma_t & = \overline{\frak{H}}^f\left(\frac{d}{dt}\,\big(\varphi^\sigma_t,e^\sigma_t\big), f^\sigma_t\right).
\end{align*}
Here again, one can rewrite that equation as a rough differential equation driven by the canonical rough path ${\bf X}^\sigma$ above the time rescalled position process of kinetic Brownian motion in $H_0$. The continuity of the It\^o-Lyons solution map then gives the following theorem. 

\begin{theorem}   \label{ThmHomogenizationM0}
The $\mathscr{M}_0$-valued part $(\varphi^\sigma_t)$ of kinetic Brownian motion in $\mathscr{Z}$ is converging locally weakly to the projection $(\varphi_t)$ on $\mathscr{M}_0$ of a $\mathscr{Z}$-valued Brownian motion.
\end{theorem}

Here again, the dynamics of $\varphi_t^\sigma$ is non-Markovian. Note that since kinetic Brownian motion on $\mathscr{M}_0$ is defined by Cartan development, using the $L^2$ metric \eqref{EqL2Metric}, the $L^2$-size of $\dot\varphi_t^\sigma$ is equal to the $L^2$-norm of $v^\sigma_t$. The metric being right invariant on the group $\mathscr{M}_0$, the Eulerian velocity 
$$
u_t^\sigma := \dot \varphi^\sigma_t\circ(\varphi^\sigma_t)^{-1},
$$
also has the same $L^2$-norm as $v^\sigma_t$. The latter is not preserved a priori; neither is the $H^s$-norm of $u^\sigma_t$, as mentioned above after Proposition \ref{PropMalliavin}.

Denote by $Q^0$ the quadratic form on $H^s(TM)$, with matrix 
$$\textrm{diag}\big(\vert\lambda_n\vert^{-s}\big)_{n\geq 0}, 
$$
in the orthonormal basis of $H^s(TM)$ associated with the eigenvector decomposition \eqref{EqEigenspaceDecomposition} for $-\Delta$ on $L^2(TM)$. For each $v$ in the unit sphere $S$ of $H^s(TM)$, one has $Q^0(v) = \|v\|_{L^2}^2$, and
$$
\|v\|_{L^2}^2 \leq \lambda_0^{-s}\|v\|_{H^s}.
$$ 
Since the $S$-valued diffusion $(v_t^\sigma)$ is ergodic, each component $(v_t^\sigma)_n$ of $v_t^\sigma$, in the decomposition \eqref{EqEigenspaceDecomposition}, is an ergodic process in the interval $\big(-\lambda_n^{-s/2}, \lambda_n^{-s/2}\big)$. The squared $L^2$-norm of $v_t^\sigma$ is also an ergodic process in the interval $(0,\lambda_0^{-s})$. It has invariant measure the image of a constant multiple of the measure with density $1/\|u\|$ with respect to the Gaussian measure in $H$ with covariance $\iota_0\iota_0^*$, by the map 
$$ 
u\in H \mapsto Q^0\big(u/\|u\|\big),
$$
from Proposition \ref{PropInvariantMeasure}. This is the invariant measure of the squared $L^2$-norm of the Eulerian velocity process $u_t^\sigma$. We emphasize that this invariant measure is independent of the interpolation parameter $\sigma\in (0,\infty)$. We record part of these facts in the following statement.

\begin{corollary}
Fix $\sigma\in (0,\infty)$. The $L^2$-norm of the velocity field $u^\sigma$ of kinetic Brownian motion is an ergodic process taking values in the interval $(0,\lambda_0^{-s})$, with invariant probability measure the image of a constant multiple of the measure with density $1/\|u\|$ with respect to the Gaussian measure in $H$ with covariance $\iota_0\iota_0^*$, by the map 
$$ 
u\in H \mapsto Q^0\big(u/\|u\|\big).
$$
\end{corollary}

\medskip

It is desirable to study the homogenization problem for other intrinsically randomly perturbated partial differential equations of geometric nature, such as the KdV, (modified) Camassa-Holm equations, or equations with non-local inertia operator, such as the modified Constantin-Lax-Majda equation \cite{KolevSurvey}. The core technical problem, from the geometric/analytic point of view, is the definition of Cartan development map as the solution map of an ordinary differential equation driven by sufficiently regular vector fields on the configuration space. We took advantage, in the present $L^2$ setting, of the `pointwise' character of the associated geometric objects to recast things in terms of the $F$ bundles of Section \ref{SubsectionConfigurationSpace}. One may have to proceed differently for other weak metrics. We expect the homogenization results proved in Theorem \ref{ThmHomogenizationM} and Theorem \ref{ThmHomogenizationM0} to have analogues in the setting of the strong, complete, Riemannian metrics of \cite{BruverisVialard}. Global in time existence results for kinetic Brownian motion and its limit Brownian motion are expected. We leave these questions for a forthcoming work.

We worked here in the Sobolev setting to make things easier and concentrate on the probabilistic problems, and the implementation of the rough path approach in this infinite dimensional setting. It is a natural question to ask whether one can run the analysis in the Fr\'echet setting of smooth diffeomorphisms of $M$, asking for preservation of the regularity of the initial condition and velocity, as in Ebin-Marsden seminal work -- Section 12 in \cite{EbinMarsden}, under proper assumptions on the noise.

\appendix

\section{Cartan development in $\mathscr{M}_0$}
\label{AppendixCartanM0}

We prove in this Appendix that Cartan's development system \eqref{EqParallelTransportM0} on $\mathscr{M}_0$ can be recast as an ordinary differential equation in $H^s\big(F^{(e)}\big)\times {\sf L}\big(H^s(TM)\big)$, driven by a smooth vector field. It has, as a consequence, a unique solution, up to a possibly finite explosion time.

\medskip

Let $\overline{P} : T\mathscr{M}\rightarrow T\mathscr{M}$, stand for a smooth vector bundle  morphism that coincides with the Hodge projector $P$ from \eqref{EqHodgeProjector} on $T\mathscr{M}_0$. The existence of such a map follows from the following elementary partition of unity result.

\begin{proposition}
\label{thm.partitionsofunity}
Let $(\mathcal O_i)_{i\in I}$ be an open cover of $\mathscr{M}$. Then there exists a smooth partition of unity subordinated to $(\mathcal O_i)_{i\in I}$.
\end{proposition}

Set 
\begin{align*}
\overline{\frak{H}}^f : T\Hs^s\big(F^{(e)}\big) \times {\sf L}\big(H^s(TM)\big) &\rightarrow T {\sf L}\big(H^s(TM)\big)   \\
\left(\frac{d}{dt}_{\big| t=0}\big(\varphi_t(\cdot),e_t(\cdot)\big),\, f\right) &\mapsto \frac{d}{dt}_{\big| t=0}\left({\bf X} \mapsto  e_t^{-1}\Big(\overline{P}\big(e_t(f({\bf X}))\big)\Big)\right).
\end{align*}
The letter $\bf X$ stands for a generic element of $H^s(TM)$, and 
$$
T{\sf L}\big(H^s(TM)\big) = {\sf L}\big(H^s(TM)\big).
$$ 
We give the details of the following elementary result.

\begin{lemma}
The map $\overline{\frak{H}}^f$ is well-defined and smooth.
\end{lemma}

\begin{proof}
It is enough to prove that the map
\begin{align*}
H^s(F^{(e)})\times {\sf L}\big(H^s(TM)\big) &\rightarrow {\sf L}\big(H^s(TM)\big)   \\
\Big(\big(\varphi(\cdot), e(\cdot)\big), f\Big) &\mapsto \left({\bf X}\mapsto e^{-1}\Big(\overline{P}\big(e(f(\bf X))\big)\Big)\right)
\end{align*}
is smooth. Since the map
\begin{align*}
H^s(F^{(e)})\times {\sf L}\big(H^s(TM)\big)\times H^s(TM) &\rightarrow H^s(TM)   \\
\Big(\big(\varphi(\cdot), e(\cdot)\big), f, {\bf X}\Big) & \mapsto e^{-1}\Big(\overline{P}\big(e(f(\bf X))\big)\Big)
\end{align*}
is smooth, the problem reduces to the following question. Let a Banach manifold $A$ and a Hilbert space $H$, be given together with a smooth map $F : A\times H\rightarrow H$, that is linear with respect to its second argument. Denote by $a$ and $b$ generic elements of $A$. Prove that the curryfication $\mathrm{Cur}\,F:a\in A\mapsto F(a,\cdot)\in {\sf L}(H)$ is well-defined and smooth.

\smallskip

Write $d$ for the differential operator. We show that $d(\mathrm{Cur}\,F)=\mathrm{Cur}\,(\partial_aF)$. This will be enough, since we can then bootstrap the construction to show that $d^n(\mathrm{Cur}\,F)=\mathrm{Cur}\,(\partial_a^nF)$, is differentiable for any $n$. Because the result is local, we can assume without loss of generality that $A$ an open set of a Banach space. Fix $a\in M$, and let $\mathcal U\times B(0,\eps)$ be a convex neighbourhood of $(a,0)$ in $A\times H$, such that $\|\partial_a^2F\|_\infty<1+\|\partial_a^2F(a,0)\|$. Then for all $b\in\mathcal U$ and $|w|<1$, one has
$$
\Big| F(b,w)-F(a,w) - \partial_a F(a,w)(b-a) \Big| \leq \frac{|b-a|^2}2\|\partial_a^2F\|_\infty \, |w|/\epsilon.
$$
The conclusion follows from the fact that we have in particular the estimate
$$
\Big\|\mathrm{Cur}F(b)-\mathrm{Cur}F(a) - \mathrm{Cur}(\partial_a F)(a;b-a)\Big\| \leq c\,|b-a|^2,
$$
for a positive constant $c$ independent of $b$.
\end{proof}

Choose now a $\mathcal C^1$ path $({\bf X}_t)$ with values in $T_\textrm{Id}\mathscr{M}_0$, and zero initial condition. Let $\big((\varphi_t, e_t), f_t)$ be the solution in $H^s(F^{(e)})\times {\sf L}\big(H^s(TM)\big)$ of the equation
\begin{equation}\begin{split}   \label{EqProofCartanM0}
\frac{d}{dt}\,(\varphi_t, e_t) & = \overline{\frak{H}}^e\Big(\varphi_t, e_t ; e_t\big(f_t(\dot{\bf X}_t)\big)\Big),   \\
\frac{d}{dt}\,f_t & = \overline{\frak{H}}^f\left(\frac{d}{dt}\,(\varphi_t, e_t), f_t\right),
\end{split}\end{equation}
with initial condition $e_0=\Id_{T\M}$, and $f_0 = \Id_{H^s(TM)}$. Since the vector field $(\overline{\frak{H}}^e, \overline{\frak{H}}^f)$ is smooth, equation \eqref{EqProofCartanM0} is locally well-posed, possibly up to a finite explosion time $\zeta$.

\begin{theorem}
\label{ThmCartanM0}
The path $(\varphi_t)$ takes values in $\mathscr{M}_0$, and coincides with the Cartan development of $({\bf X}_t)$. We further have $\dot \varphi_t = e_t\big(f_t(\dot {\bf X}_t)\big)$, so the dynamics \eqref{EqProofCartanM0} does not depend on the extension $\overline{P}$ of the Hodge projector $P$ used in the definition of $\overline{\frak{H}}^f$.
\end{theorem}

\begin{proof}
Let ${\bf Y}\in T_\textrm{Id}\mathscr{M}_0$, be a fixed divergence-free vector field on $M$. We need to show that
$$ 
\Nabla^0_{\dot \varphi_t} e_t({\bf Y}) = 0,
$$
on the whole time interval $[0,\zeta)$.  From Proposition \ref{PropParallelTransportM0}, this is equivalent to showing that we have
$$
\frac{d}{dt}\, \Big(\varphi_t, e_t\big(f_t({\bf Y})\big)\Big) = dP\Big(\overline{\frak{H}}^{(v)}\Big(\varphi_t, e_t\big(f_t({\bf Y})\big) ; \dot \varphi_t\Big)\Big). 
$$
Look at the function 
$$
(\varphi, e, {\bf Z})\mapsto \big(\varphi, e({\bf Z})\big),
$$
from $H^s(F^{(e)})\times T_{\textrm{Id}}\mathscr{M}$ to $H^s(F^{(v)})$, and set 
$$
\frak{F} := \partial_{(\varphi,e)}\Big\{(\varphi, e, {\bf Z})\mapsto \big(\varphi, e({\bf Z})\big)\Big\}.
$$ 
We have
\begin{equation*}\begin{split}
\frac{d}{dt}\,&\Big(\varphi_t, e_t\big(f_t({\bf Y})\big)\Big)    \\
&= \frak{F}\left(\frac{d}{dt}\,(\varphi_t,e_t), f_t({\bf Y})\right) - \frak{F}\left(\frac{d}{dt}\,(\varphi_t,e_t), e_t^{-1}\Big(\overline{P}\big(e_t(f_t({\bf Y}))\big)\Big)\right)    \\
&\quad+ d\overline{P}\left(\frak{F}\Big(\frac{d}{dt}\,(\varphi_t,e_t), f_t({\bf Y})\Big)\right). 
\end{split}\end{equation*}
We prove that $e_t(\bf Y)$ is divergence-free. Define for that purpose the subset $I\subset[0,\zeta)$ of times $t$ such that $e_t(\bf Z)$ is divergence-free for all ${\bf Z}\in T_\textrm{Id}\mathscr{M}_0$, and $\varphi_t$ preserves the volume form. It is a non-empty closed subset of $[0,\zeta)$. Fix $t_0\in\ I$. It suffices to prove that $t_0$ is in the interior of $I$ for a well-chosen extension $\widehat{P}$ of $P$, possibly different from $\overline{P}$. We choose for $\widehat{P}$ any smooth extension of $P$ defined on a neighbourhood of $\varphi_{t_0}$, such that $\widehat{P}\circ\widehat{P} = \widehat{P}$. Set $\widehat{Q} := \Id - \widehat{P} : T\mathscr{M}\to T\mathscr{M}$, so for a fixed ${\bf Z}\in T_\textrm{Id}\mathscr{M}_0$, the quantity 
$$
Z_t := \widehat{Q}\big(e_t(f_t({\bf Z}))\big)
$$ 
satisfies the equation
\begin{equation*}\begin{split}
\frac{d}{dt}\,Z_t &= d\widehat{Q}\left(\frak{F}\left(\frac{d}{dt}\,(\varphi_t,e_t), e_t^{-1}\Big(\widehat{Q}(e_t[f_t({\bf Z})])\Big)\right)\right)   \\
&= d\widehat{Q}\left(\frak{F}\Big(\frac{d}{dt}\,(\varphi_t,e_t), e_t^{-1}(Z_t)\Big)\right).
\end{split}\end{equation*}
This differential equation satisfies the classical Picard-Lindel\"of assumptions, so it has a unique solution with given initial condition. Since $Z_0=0$ and the constant zero vector field is a solution to the equation, $Z_t$ is identically zero, and $e_t({\bf Z})$ is divergence-free. 

This holds true for any $\bf Z$, in a time interval independent of $\bf Z$. It follows in particular that $\dot \varphi_t = e_t\big(f_t(\dot {\bf X}_t)\big)$ is locally divergence-free, and $\varphi_t$ preserves the volume form, in a neighbourhood of the time $t_0$. The interval $I$ is thus both closed and open, so $I=[0,\zeta)$. The statement of Theorem \ref{ThmCartanM0} follows, since $P\big(e_t(f_t({\bf Y}))\big) = e_t\big(f_t(\bf Y)\big)$, so we get
\begin{equation*}\begin{split}
\frac{d}{dt}\, \Big(\varphi_t,e_t\big(f_t({\bf Y})\big)\Big) &= d\overline{P}\left(\frak{F}\Big(\frac{d}{dt}\,(\varphi_t,e_t), f_t({\bf Y})\Big)\right)   \\
&= d\overline{P}\left(\frac{d}{ds}_{\big|s=t} \Big(\varphi_s, e_s\big(f_t({\bf Y})\big)\Big)\right)   \\
&= dP\Big(\overline{\frak{H}}^{(v)}\Big(\varphi_t, e_t\big(f_t({\bf Y})\big) ; \dot \varphi_t\Big)\Big),
\end{split}\end{equation*}
using Proposition \ref{PropCartanM0Argument} in the last equality. 
\end{proof}

\bibliographystyle{alpha}
\bibliography{refs}

\newcommand{\etalchar}[1]{$^{#1}$}
\begin{thebibliography}{BdLHLT19}

\bibitem[ABT15]{ABT}
J.~Angst, I.~Bailleul, and C.~Tardif.
\newblock Kinetic brownian motion on {R}iemannian manifolds.
\newblock {\em Elec. J. Probab.}, 20(110):1--40, 2015.

\bibitem[AF07]{AngstFranchi}
J.~Angst and J.~Franchi.
\newblock A central limit theorem for a class of relativistic diffusions.
\newblock {\em J. Math. Phys.}, 48(3):083101, 2007.

\bibitem[AHK12]{VassiliHilbert}
S.~Albeverio, A.~Hilbert, and V.N. Kolokoltsov.
\newblock Uniform asymptotic bounds for the heat kernel and the trace of a
  stochastic geodesic flow.
\newblock {\em Stochastics}, 84(2-3):315--333, 2012.

\bibitem[AK98]{ArnoldKhesin}
V.I. Arnold and B.A. Khesin.
\newblock {\em Topological methods in hydrodynamics}, volume 125 of {\em
  Applied Mathematical Sciences}.
\newblock Springer, 1998.

\bibitem[AM02]{MalliavinAirault}
H.~Airault and P.~Malliavin.
\newblock Quasi-invariance of {B}rownian measures on the group of circle
  homeomorphisms and infinite-dimensional {R}iemannian geometry.
\newblock {\em J. Funct. Anal.}, 196:395--446, 2002.

\bibitem[AR02]{AiraultRen}
H.~Airault and J.~Ren.
\newblock Modulus of continuity of the canonic {B}rownian ''on the group of
  diffeomorphisms of the circle''.
\newblock {\em J. Funct. Anal.}, 196:395--446, 2002.

\bibitem[Arn66]{Arnold}
V.I. Arnold.
\newblock Sur la g\'eom\'etrie diff\'erentielle des groupes de {L}ie de
  dimension infinie et ses application \`a l'hydrodynamique des fluides
  parfaits.
\newblock {\em Ann. Inst. Fourier}, 16(1):319--361, 1966.

\bibitem[AS89]{ArakelyanSavvidy}
T.A. Arakelyan and G.K. Savvidy.
\newblock Geometry of a group of area-preserving diffeomorphisms.
\newblock {\em Physics Letters B}, 223(1), 1989.

\bibitem[Bai]{BailleulLN}
I.~Bailleul.
\newblock A flow-based approach to rough differential equations.
\newblock {\em
  https://perso.univ-rennes1.fr/ismael.bailleul/files/M2Course.pdf}, pages
  1--63.

\bibitem[Bai10]{BailleulRelDiff}
I.~Bailleul.
\newblock A stochastic approach to relativistic diffusions.
\newblock {\em Ann. Inst. Henri Poincar\'e Probab. Stat.}, 46(3):760--795,
  2010.

\bibitem[Bai15a]{BailleulSeminaire}
I.~Bailleul.
\newblock Flows driven by {B}anah space valued rough paths.
\newblock {\em S\'eminaire Probab.}, XLVI:195--205, 2015.

\bibitem[Bai15b]{BailleulFlows}
I.~Bailleul.
\newblock Flows driven by rough paths.
\newblock {\em Rev. Mat. Iberoamericana}, 31(3):901--934, 2015.

\bibitem[Bau14]{BaudoinLNEMS}
F.~Baudoin.
\newblock {\em Diffusion processes and stochastic calculus}.
\newblock EMS Textbooks in Mathematics. European Mathematical Society, 2014.

\bibitem[BC17]{BailleulCatellier}
I.~Bailleul and Catellier.
\newblock Rough flows and homogenization in stochastic turbulence.
\newblock {\em J. Diff. Eq.}, 263(8):4894--4928, 2017.

\bibitem[BCD11]{BCD}
H.~Bahouri, J.-Y. Chemin, and R.~Danchin.
\newblock {\em Fourier analysis and nonlinear partial differential euqations.},
  volume 343 of {\em Grundlehren des mathemtischen Wissenschaften}.
\newblock Springer, 2011.

\bibitem[BdLHLT19]{KIWFormulaHolm}
A.~Bethencourt~de L\'eon, D.~Holm, E.~Luesink, and S.~Takao.
\newblock Implications of {K}unita-{I}t\^o-{W}entwell formula for k-forms in
  stochastic fluid dynamics.
\newblock {\em arXiv}, 1903.07201v1:1--25, 2019.

\bibitem[BF12]{BailleulFranchi}
I.~Bailleul and J.~Franchi.
\newblock Non-explosion criteria for relativistic diffusions.
\newblock {\em Ann. Probab.}, 40(3):2168--2196, 2012.

\bibitem[Bis05]{Bismut05}
J.-M. Bismut.
\newblock The hypoelliptic {L}aplacian on the cotangent bundle.
\newblock {\em J. Am. Math. Soc.}, 18(2):379--476, 2005.

\bibitem[Bis11]{BismutOrbital}
J.-M. Bismut.
\newblock {\em Hypoelliptic {L}aplacian and Orbital Integrals}.
\newblock Annals of Mathematical Studies. Princeton University Press, 2011.

\bibitem[Bis15]{Bismut}
J.-M. Bismut.
\newblock Hypoelliptic {L}aplacian and probability.
\newblock {\em J. Math. Soc. Japan}, 67(4):1317--1357, 2015.

\bibitem[Bis16]{BismutEta}
J.-M. Bismut.
\newblock Eta invariants and the hypoelliptic {L}aplacian.
\newblock {\em arXiv:1603.05103}, pages 1--155, 2016.

\bibitem[BV20]{BruverisVialard}
M.~Bruveris and F.-X. Vialard.
\newblock On completeness of groups of diffeomorphisms.
\newblock {\em To appear in J. Europ. Math. Soc.}, pages 1--43, 2020.

\bibitem[BVW17]{BirrellHottovyVolpeWehr}
S.~Birrell, J. ad~Hottovy, G.~Volpe, and J.~Wehr.
\newblock Small mass limit of a langevin equation on a manifold.
\newblock {\em Ann. Henri Poincar\'e}, 18(2):707--755, 2017.

\bibitem[BW18]{BirrellWehr}
J.~Birrell and J.~Wehr.
\newblock Langevin equations in the small-mass limit: Higher order
  approximations.
\newblock {\em arXiv:1809.01724}, pages 1--38, 2018.

\bibitem[Car01]{Cartan}
E.~Cartan.
\newblock {\em Riemannian geometry in an orthogonal frame}.
\newblock World Scientific, 2001.

\bibitem[CFH18]{CrisanFlandoliHolm}
D.~Crisan, F.~Flandoli, and D.~Holm.
\newblock Solution properties of a 3d stochastic {E}uler fluid equation.
\newblock {\em J. Nonlinear Sci.}, 242:1--58, 2018.

\bibitem[CFK{\etalchar{+}}19]{CFKMZh}
I.~Chevyrev, P.K. Friz, A.~Korepanov, I.~Melbourne, and H.~Zhang.
\newblock Deterministic homogenization for discrete-time fast-slow systems
  under optimal moment assumptions.
\newblock {\em arXiv:1903.10418v1}, pages 1--24, 2019.

\bibitem[CFM07]{CFM}
A.-B. Cruzeiro, F.~Flandoli, and P.~Malliavin.
\newblock Brownian motion on volume preserving diffeomorphisms group and
  existence of global solutions of 2d stochastic {E}uler equation.
\newblock {\em J. Funct. Anal.}, 242:304--326, 2007.

\bibitem[CHR18]{CruzeiroHolmRatiu}
A.-B. Cruzeiro, D.~Holm, and T.~Ratiu.
\newblock Momentum maps and stochastic {C}lebsch action principles.
\newblock {\em Comm. Math. Phys.}, 357(2):873--912, 2018.

\bibitem[CLL07]{LyonsStFlour}
M.~Caruana, Th. L\'evy, and T.J. Lyons.
\newblock {\em Differential equations driven by rough paths}, volume 1908.
\newblock 2007.

\bibitem[Cun17]{Cuny}
Ch. Cuny.
\newblock Invariance principles under the {M}axwell-{W}oodroofe condition in
  {B}anach spaces.
\newblock {\em Ann. Probab.}, 45(3):1578--1611, 2017.

\bibitem[CW16]{CassWeidner}
T.~Cass and M.P. Weidner.
\newblock Tree algebras over topological vector spaces in rough path theory.
\newblock {\em arXiv:1604.07352v2}, pages 1--25, 2016.

\bibitem[DH18]{DrivasHolmCirculationCharact}
D.~Drivas and D.~Holm.
\newblock Circulation and energy theorem preserving stochastic fluids.
\newblock {\em arXiv}, 1808.05308v1:1--26, 2018.

\bibitem[DM03]{DedeckerMerlevede}
J.~Dedecker and F.~Merlev\`ede.
\newblock The conditional central limit theorem in {H}ilbert spaces.
\newblock {\em Stoch. Proc. Appl.}, 108:229--262, 2003.

\bibitem[Dro17]{Drouot}
A.~Drouot.
\newblock Stochastic stability of {P}ollicott-{R}uelle resonances.
\newblock {\em Comm. Math. Phys.}, 356:357--396, 2017.

\bibitem[Dud66]{Dudley}
R.M. Dudley.
\newblock Lorentz-invariant {M}arkov processes in relativistic phase space.
\newblock {\em Ark. Mat}, 6:241--268, 1966.

\bibitem[EM69]{EbinMarsden}
D.~G. Ebin and J.~E. Marsden.
\newblock Groups of diffeomorphisms and the solution of the classical {E}uler
  equations for a perfect fluid.
\newblock {\em Bull. Amer. Math. Soc.}, 75:962--967, 1969.

\bibitem[Fan02]{Fang}
Sh. Fang.
\newblock Canonical brownian motion on the diffeomorphism group of the circle.
\newblock {\em J. Funct. Anal.}, 196:162--179, 2002.

\bibitem[Fer70]{Fernique}
X.~Fernique.
\newblock Int\'egrabilit\'e des vecteurs gaussiens.
\newblock {\em C. R. Acad. Sci. Paris}, 270(25):1698--1699, 1970.

\bibitem[FGL13]{FrizGassiatLyons}
P.~Friz, P.~Gassiat, and T.J. Lyons.
\newblock Physical {B}rownian motion in magnetic field as rough path.
\newblock {\em Transactions of the American Mathematical Society}, 367(11),
  2013.

\bibitem[FH14]{FrizHairer}
P.~K. Friz and M.~Hairer.
\newblock {\em A course on rough paths}.
\newblock Universitext. Springer, Cham, 2014.
\newblock With an introduction to regularity structures.

\bibitem[FLJ07]{FranchiLeJanRelDiff1}
J.~Franchi and Y.~Le~Jan.
\newblock Relativistic diffusions and schwarzschild geometry.
\newblock {\em Comm. Pure Appl. Math.}, 60(2):187--251, 2007.

\bibitem[FLJ11]{FranchiLeJanRelDiff2}
J.~Franchi and Y.~Le~Jan.
\newblock Curvature diffusions in general relativity.
\newblock {\em Comm. Math. Phys.}, pages 307--351, 2011.

\bibitem[GBH17]{GayBalmazHolm}
F.~Gay-Balmaz and D.~Holm.
\newblock Stochastic geometric models with non-stationary spatial correlations
  in {L}agrangian fluid flows.
\newblock {\em J. Nonlinear Sci.}, 28(3):873--904, 2017.

\bibitem[Gli11]{Gliklikh}
Y.~Gliklikh.
\newblock {\em Global and Stochastic Analysis with Applications to Mathematical
  Physics}.
\newblock Theoretical and Mathematical Physcis. Springer, 2011.

\bibitem[Hai12]{LNHairerSPDE}
M.~Hairer.
\newblock An introduction to stochastic {PDE}s.
\newblock {\em http://www.hairer.org/notes/SPDEs.pdf}, 2012.

\bibitem[Hol15]{Holm15}
D.~Holm.
\newblock Variational principles for stochastic fluid dynamics.
\newblock {\em Proc. R. Soc. A}, 471:20140963, 2015.

\bibitem[HV16]{HerzogHottovyVolpe}
S.~Herzog, D. ad~Hottovy and G.~Volpe.
\newblock The small-mass limit for {L}angevin dynamics with unbounded
  coefficients and positive friction.
\newblock {\em J. Stat. Phys.}, 163(3):659--673, 2016.

\bibitem[KM16]{KellyMelbourne1}
D.~Kelly and I.~Melbourne.
\newblock Smooth approximation of stochastic differential equations.
\newblock {\em Ann. Probab.}, 44(1):479--520, 2016.

\bibitem[KM17]{KellyMelbourne2}
D.~Kelly and I.~Melbourne.
\newblock Deterministic homogenization for fast-slow systems with chaotic
  noise.
\newblock {\em J. Funct. Anal.}, 272(10):4063--4102, 2017.

\bibitem[Kol00]{Kolokoltsov}
V.N. Kolokoltsov.
\newblock {\em Semiclassical analysis for diffusions and stochastic processes}.
\newblock Number 1724 in Lecture notes in Mathematics. Springer, 2000.

\bibitem[Kol17]{KolevSurvey}
B.~Kolev.
\newblock Local well-posedness of the {EPD}iff equation: A survey.
\newblock {\em J. Geom. Mech.}, (9(2)):167--189, 2017.

\bibitem[Kun90]{Kunita}
H.~Kunita.
\newblock {\em Stochastic flows}, volume~24 of {\em Probability Theory and
  Stochastic Modelling}.
\newblock Cambridge Univ. Press, 1990.

\bibitem[Li12]{XueMei0}
X.-M. Li.
\newblock Effective diffusions with intertwined structures.
\newblock {\em arXiv}, 1204.3250v1:1--33, 2012.

\bibitem[Li16]{XueMei1}
X.-M. Li.
\newblock Random perturbation to the geodesic equation.
\newblock {\em Ann. Probab.}, 44(1):544--566, 2016.

\bibitem[Li18]{XueMei2}
X.-M. Li.
\newblock Homogenisation on homogeneous spaces.
\newblock {\em J. Math. Soc. Japan}, 70(2):519--572, 2018.

\bibitem[LWL19]{LimWehrLewenstein}
S.H. Lim, J.~Wehr, and M.~Lewenstein.
\newblock Homogenization for generalized langevin equations with applications
  to anomalous diffusion.
\newblock {\em arXiv:1902.06496}, pages 1--63, 2019.

\bibitem[Lyo94]{LyonsYoung}
T.J. Lyons.
\newblock Differential equations driven by rough signals. {I}. an extension of
  an inequality of {L. C. Y}oung.
\newblock {\em Math. Res. Lett.}, 1(4):451--464, 1994.

\bibitem[Lyo98]{Lyons98}
T.J. Lyons.
\newblock Differential equations driven by rough signals.
\newblock {\em Rev. Mat. Iberoamericana}, 14(2), 1998.

\bibitem[Mal99]{Malliavin}
P.~Malliavin.
\newblock The canonic diffusion above the diffeomorphism group of the circle.
\newblock {\em C.R. Acad. Sci. Paris}, 329:325--329, 1999.

\bibitem[Pal68]{Palais}
R.~S. Palais.
\newblock {\em Foundations of global non-linear analysis}.
\newblock W. A. Benjamin, Inc., New York-Amsterdam, 1968.

\bibitem[Per18]{PierreFiniteDim}
P.~Perruchaud.
\newblock Homogenisation for anisotropic kinetic random motion.
\newblock {\em arXiv}, page 1811.08415, 2018.

\bibitem[PZ95]{PeszatZabczyk}
S.~Peszat and J.~Zabczyk.
\newblock Strong feller property and irreducibility for diffusions on hilbert
  spaces.
\newblock {\em Ann. Probab.}, 23(1):157--172, 1995.

\bibitem[Ros97]{Rosenberg}
S.~Rosenberg.
\newblock {\em The Laplacian on a Riemannian manifold}, volume~31 of {\em
  Student texts}.
\newblock London Mathematical Society, 1997.

\bibitem[She16]{ShuShen}
Sh. Shen.
\newblock Laplacien hypoelliptique, torsion analytique, et th\'eor\`eme de
  {C}heeger-{M}\"uller.
\newblock {\em J. Funct. Anal.}, (270):2817--2999, 2016.

\bibitem[Smo07]{Smolentsev}
N.K. Smolentsev.
\newblock Diffeomorphism group of compact manifolds.
\newblock {\em Journal of Mathematical Sciences}, 146(6):6213--6312, 2007.

\bibitem[Sol95]{Soloveitchik}
M.R. Soloveitchik.
\newblock Focker-{P}lanck equation on a manifold. {E}ffective diffusion and
  spectrum.
\newblock {\em Potential Analysis}, (4):571--593, 1995.

\bibitem[Str93]{Stroock}
D.~W. Stroock.
\newblock {\em Probability theory, an analytic view}.
\newblock Cambridge University Press, Cambridge, 1993.

\bibitem[WZ10]{WangZhang}
F.Y. Wang and T.~Zhang.
\newblock Strong feller property and irreducibility for diffusions on hilbert
  spaces.
\newblock {\em J. Math. Ana. Appl.}, 365:1--11, 2010.

\bibitem[Yos97]{Yoshida}
K.~Yoshida.
\newblock Riemannian curvature on the group of area-preserving diffeomorphisms
  (motions of fluid) of $2$-sphere.
\newblock {\em Physics D}, 100:377--389, 1997.

\bibitem[You36]{Young}
L.C. Young.
\newblock An inequality of the {H}\"older type, connected with {S}tieltjes
  integration.
\newblock {\em Acta Math.}, 67:251--282, 1936.

\end{thebibliography}

\end{document}